\documentclass[11pt,reqno]{amsart}
\usepackage{amsmath, amsthm, amssymb}
\usepackage{mathrsfs}
\usepackage{array}
\usepackage{caption}

\evensidemargin \oddsidemargin
\newcommand{\mmod}[1]{\,\,(\text{\rm mod}\,\, #1)}

\def\bfd{{\boldsymbol d}}

\def\bfn{{\boldsymbol n}}

\newtheorem{thm}{Theorem}

\newtheorem{lem}{Lemma}

\newtheorem{conj}{Conjecture}

\newtheorem{prop}{Proposition}

\numberwithin{equation}{section} \numberwithin{thm}{section}
\numberwithin{lem}{section} \numberwithin{problem}{section}
\numberwithin{cor}{section}

\begin{document}
\title{Twisted mixed moments of the Riemann zeta function}
\author[Javier Pliego]{Javier Pliego}
\address{Department of Mathematics, KTH Royal Institute of Technology, Lindstedtsv\"agen 25,
10044 Stockholm, Sweden}

\email{javierpg@kth.se}

\begin{abstract} We analyse a collection of twisted mixed moments of the Riemann zeta function and establish the validity of asymptotic formulae comprising on some instances secondary terms of the shape $P(\log T) T^{C}$ for a suitable constant $C<1$ and a polynomial $P(x)$. Such examinations are performed both unconditionally and under the assumption of a weaker version of the $abc$-conjecture.
\end{abstract}
\maketitle

\section{Introduction}
Moments of the Riemann-zeta function, namely
$$M_{k}(T)=\int_{0}^{T}\lvert \zeta(1/2+it)\rvert^{2k}dt$$ for $k\in\mathbb{N}$ are a central object of study in the analytic theory of numbers due to the numerous applications that results concerning these have appertaining to both the properties of zeta and the distribution of prime numbers. After the sucessful asymptotic evaluations of both the second and the fourth moment due to Hardy-Littlewood \cite{Hard} and Ingham \cite{Ing} respectively, these in turn being sharpened in subsequent work of numerous authors (see \cite{Bou6,Heath,Ivi}) and taking the shape
\begin{equation}\label{Mkk}M_{k}(T)=TP_{k^{2}}(\log T)+O(T^{1-\delta})\ \ \ \ \ \ \ \ \ \ \end{equation}
for $k=1,2$, wherein $P_{k^{2}}(x)$ is a degree-$k^{2}$ polynomial and $0<\delta<1$ is a fixed constant, no unconditional asymptotic evaluation of higher moments has been achieved. It is nonetheless worth mentioning the conjectural work of Conrey et al. \cite{Conr} delivering (\ref{Mkk}) with $\delta=1/2-\varepsilon$ for $k\geq 3$. The analysis of twisted moments of $L$-functions, which we define in this setting for $T>1$ and $k\in \mathbb{N}$ by
\begin{equation}\label{ostti}\int_{0}^{T}\lvert A(1/2+it)\rvert^{2}\lvert \zeta(1/2+it)\rvert^{2k}dt,\ \ \ \ \  \ \ \ \ \  \text{wherein}\ \ A(s)=\sum_{n\leq T^{\vartheta}}\frac{a_{n}}{n^{s}}\end{equation} is a Dirichlet polynomial with $\vartheta>0$ being a fixed constant and $a_{n}\ll n^{\varepsilon}$, delivers consequences concerning upper and lower bounds for $L$-functions and properties about the distribution of the zeroes, a brief account of the history of such a problem being provided at a later point in the introduction. 

We draw the reader's attention to \cite{Pli}, wherein for coefficients $a,b,c\in \mathbb{N}$ satisfying $a<c\leq b$ the asymptotic evaluation 
\begin{equation}\label{eleme}\int_{0}^{T}\zeta(1/2+ait)\zeta(1/2-bit)\zeta(1/2-cit)dt=\sigma_{a,b,c}T+O(T^{1-\delta})\end{equation} for some fixed $\delta>0$ was delivered, the corresponding error term being sharpened therein under the assumption of the following consequence of the abc-conjecture (see \cite[Lemma 2.2]{Pli}).
\begin{conj}\label{conj11}
Let $a,b,c\in\mathbb{N}$ be fixed natural numbers and $n_{1},n_{2},n_{3}\in\mathbb{N}$. Denote $$D=n_{1}^{a}-n_{2}^{b}n_{3}^{c}.$$ Then, if $D\neq 0$ one has the lower bound
$$\lvert D\rvert\gg n_{1}^{a-1-\varepsilon}n_{2}^{-1}n_{3}^{-1}.$$ 
\end{conj} 
Inspired by the preceding paragraphs and the above considerations, we introduce for $\theta>0$ the Dirichlet polynomial
\begin{equation}\label{diri}D_{\theta}(s)=\sum_{n\leq (t/2\pi)^{\theta}}n^{-s}\ \ \ \ \ \ \ \ \ \ s=\sigma+it,\end{equation} and for triples $(a,b,c)\in\mathbb{N}^{3}$ consider 
\begin{equation}\label{prrosa}M_{a,b,c}^{\theta}(T)=\int_{0}^{T}D_{\theta}(1/2+ait)\zeta(1/2-bit)\zeta(1/2-cit)dt.\end{equation}
We should note that in the examination of (\ref{prrosa}), extra terms arise which would otherwise vanish in its counterpart of (\ref{eleme}) and which encode the arithmetic structure of certain underlying diophantine equation. When either $a\mid b$ or $a\mid c$ we are then able to detect analogous lower order terms in the asymptotic evaluation of the preceding object subject to the validity of Conjecture \ref{conj11}.  
\begin{thm}\label{thmabcy}
Let $(a,b,c)\in\mathbb{N}^{3}$ such that $a<c< b$ and $a\geq 2$ with either $a\mid b$ or $a\mid c$ and $(a,b,c)=1$. Let $\theta\in\mathbb{R}$ be a fixed parameter satisfying $0<\theta<\min(c/2a,1)$. Under the assumption of Conjecture \ref{conj11} it follows that
$$M_{a,b,c}^{\theta}(T)=K_{a,b,c}T-C_{a,r,s}T^{1-\theta(1/2-a/2r)}+O(T^{1-\theta/2}+T^{1/2+\theta/2}+T^{\theta/2+3a\theta/2c+\varepsilon}),$$ where we wrote $\{r,s\}=\{b,c\}$ so that $a\mid r$ and $a\nmid s$, and where $K_{a,b,c}>1$ and $C_{a,r,s}>0$ are constants that shall be defined in (\ref{Cars}) and (\ref{Kabc}) respectively. If no divisibility relations hold and $\theta<\min(c/2a,1)$ then one instead has
$$M_{a,b,c}^{\theta}(T)=\sigma_{a,b,c}T+O(T^{1-\theta/2}+T^{1/2+\theta/2}+T^{\theta/2+3a\theta/2c+\varepsilon}),$$ wherein $\sigma_{a,b,c}>1$ is a constant that shall be defined in (\ref{sigmaa}).
\end{thm}
The topic of twisted moments was introduced inter alia by Levinson \cite{Lev} for the purpose of providing a lower bound for the proportion of the zeros of the Riemann-zeta function on the critical line. An asymptotic formula for (\ref{ostti}) when $k=1$ was obtained by Balasubramanian et al. \cite{Bala} when $\vartheta=1/2-\varepsilon$. Levinson's ideas were exploited in the paper of Conrey \cite{Con} by extending the length of the polynomial to $T^{4/7-\varepsilon}$ for a specific choice of the coefficients to the end of showing that the aforementioned proportion is at least $2/5$, thereby improving Levinson's original result.

Much in the same vein, investigations pertaining to the twisted fourth moment were initiated by Deshouillers and Iwaniec \cite{Iwa}, wherein they provided upper bounds for (\ref{ostti}) for the choice $k=2$ with $\vartheta=1/5-\varepsilon.$ Precise asymptotics for the above object comprising lower order terms were delivered in the paper of Young and Hughes \cite{You} for $\vartheta=1/11-\varepsilon$, such a parameter being superseded by $\vartheta=1/4-\varepsilon$ in the work of Bettin et al. \cite{Bett}.

The results earlier mentioned established conclusions for general coefficients $a_{n}$, our analysis in contrast being confined to particular choices. We draw the reader's attention to (\ref{Mkk}) and raise the question of whether analogous formulae in similar settings but with lower order terms substantially smaller occur, examinations investigating the validity of such rare formulae being found, inter alia, in the work concerning moments of quadratic Dirichlet $L$-functions.  The sharpening of \cite{Soun}, namely \begin{equation}\label{MRD}\mathop{{\sum_{0<d<D}}^*}L\Big(\frac{1}{2},\chi_{d}\Big)^{3}=DP_{6}(\log D)+O(D^{\theta})\ \ \ \ \ \end{equation} for some fixed $\theta<1$, wherein the above sum is over fundamental discriminants and $P_{6}(x)$ is a degree $6$ polynomial,  was accomplished by Diaconu, Goldfeld and Hoffstein \cite{Dia}, the existence of a lower order term of the shape $cD^{3/4}$ being further conjectured in the aforementioned memoir. This speculation was essentially confirmed in the work of Diaconu and Whitehead \cite{Dia4} at the cost of introducing a weight function, thereby providing such a lower term but only for a smooth version of (\ref{MRD}). We refer to \cite{Dia5} for generalisations to higher moments. 

We observe that Theorem \ref{thmabcy} thereby incorporates a new collection of such bizarre formulae, it being a noteworthy feature that the secondary terms stemming from the diagonal contribution are particularly large when either $(a,b)$ or $(a,c)$ are relatively big, and exceed the $T^{1-\theta/2}$ barrier when at least one of the divisibility relations holds. 

In the course of the proof, we shall make use of the approximate functional equation of the Riemman zeta function (see Titchmarsh \cite[(4.12.4)]{Tit}), namely
\begin{equation}\label{approxim}\zeta(1/2+it)=D_{1/2}(1/2+it)+\chi(1/2+it)D_{1/2}(1/2-it)+O(t^{-1/4}),\end{equation}
wherein we recall (\ref{diri}) and \begin{equation}\label{chichi}\chi(s)=2^{s-1}\pi^{s}\sec(s\pi/2)\Gamma(s)^{-1},\ \ \ \ \ \ \ \ \ \ s\in\mathbb{C}\setminus (2\mathbb{Z}+1),\end{equation} thereby reducing our task to that of examining several integrals of twisted Dirichlet polynomials. We shall be confronted as in \cite{Pli} with the problem of understanding the number of solutions of the equation
\begin{equation}\label{D1D2}n_{2}^{b}n_{3}^{c}-n_{1}^{a}=D\end{equation} for $D\in\mathbb{Z}$ stemming from the off-diagonal contribution. Considerations of similar problems suggesting that lower bounds for the difference $D$ with the current techniques available (see \cite{Pli}) are of the shape $\lvert D\rvert\gg (\log n_{1})^{\delta}$ for some $\delta>0$, it should not come as a surprise that in this new setting, the validity of Conjecture \ref{conj11} must be assumed in order the aforementioned secondary term, it in turn stemming from the diagonal contribution, be larger than the corresponding error term.

The relative simplicity of the off-diagonal contribution when $a=1$ permits one to refine the analysis pertaining to the general setting. We thus deduce unconditional asymptotic formulae comprising two lower order terms at the cost of accomplishing a prolix discussion involving an application of Perron's formula in conjunction with moment estimates to the end of deriving error terms of the requisite precision, the elementary methods utilised in the proof of Theorem \ref{thmabcy} failing to detect the second lower order term. 
\begin{thm}\label{prop3k4}
Let $b,c\in\mathbb{N}$ satisfying $1<c<b$. Let $\theta\in\mathbb{R}$ be a fixed parameter satisfying $0<\theta<1$. Then one has
$$M_{1,b,c}^{\theta}(T)=K_{1,b,c}T-C_{1,c,b}T^{1-\theta(1/2-1/2c)}-C_{1,b,c}T^{1-\theta(1/2-1/2b)}+O\big(T^{1-\theta/2}(\log T)^{2}+T^{1/2+\theta/2}\big).$$ If on the contrary $b=c>1$ then instead
$$M_{1,b,b}^{\theta}(T)=K_{1,b,b}T-H_{b}(\log T)T^{1-\theta(1/2-1/2b)}+O\big(T^{1-\theta/2}(\log T)^{2}+T^{1/2+\theta/2}\big),$$ the polynomial $H_{b}(x)$ being defined in (\ref{kris}). 
\end{thm}

Similarly, one may derive in the above context an unconditional asymptotic formula comprising several lower order terms in the spirit of Theorem \ref{thmabcy} for $M_{1,b,1}^{\theta}(T)$. 
\begin{thm}\label{ppop}
Let $b>1$ and let $\theta\in\mathbb{R}$ be a fixed parameter satisfying $0<\theta\leq 1/2$. Then with the above notation one has that
$$M_{1,b,1}^{\theta}(T)=TK_{b}(\log T)-C_{1,b,1}T^{1-\theta(1/2-1/2b)}+O\big(T^{1-\theta/2}(\log T)^{2}\big),$$ wherein $K_{b}(x)$ is a linear polynomial that shall be defined in (\ref{Kb}). If instead $b=1$ then 
$$M_{1,1,1}^{\theta}(T)=\frac{\theta^{2}}{2}T(\log T)^{2}+c_{1}T\log T+c_{0}T+O\big(T^{1-\theta/2}(\log T)^{2}\big),$$ the constants $c_{0},c_{1}\in\mathbb{R}$ being defined in (\ref{c1c2c3e}).
\end{thm}
We also find it worth introducing as in (\ref{ostti}) the integral
$$M_{a,a,b,c}^{\theta}(T)=\int_{0}^{T}\lvert D_{\theta}(1/2+ait)\rvert^{2}\zeta(1/2-bit)\zeta(1/2-cit)dt.$$
We shall incorporate to the above sequel a similar asymptotic evaluation of $M_{1,1,b,c}^{\theta}(T)$ by employing manoeuvres within the same circle of ideas, widening the range of $\theta$ for which one may secure the validity of such an evaluation not being the focus of this memoir. 

\begin{thm}\label{thm10}
Let $b,c\in\mathbb{N}$ satisfying $1<c<b$. Let $\theta\in\mathbb{R}$ be a fixed parameter satisfying $0<\theta\leq 1/2$. Then one has that
\begin{align*}M_{1,1,b,c}^{\theta}(T)=&TK_{b,c}(\log T)-C_{1,c,b}'T^{1-\theta(1/2-1/2c)}-C_{1,b,c}'T^{1-\theta(1/2-1/2b)}
\\
&+O\big(T^{1-\theta/2}(\log T)^{13/4}+T^{1/2+\theta}\big),
\end{align*}the constants $C_{1,b,c}',C_{1,c,b}'$ and the linear polynomial $K_{b,c}(x)$ being defined in (\ref{abf}) and (\ref{linepo}) respectively. If instead $b=c>1$ then it transpires that
\begin{align*}M_{1,1,b,b}^{\theta}(T)=&TK_{b,b}(\log T)-G_{b}(\log T)T^{1-\theta(1/2-1/2b)}
\\
&+O\big(T^{1-\theta/2}(\log T)^{13/4}+T^{1/2+\theta}\big),
\end{align*}
wherein $G_{b}(x)$ is a linear polynomial that shall be defined in (\ref{Polyn}).
\end{thm}

We also note that one may derive analogous formulae both for $M_{1,1,b,1}^{\theta}(T)$ with $b>1$ and $b=1$, the linear polynomial $K_{b,c}(\log T)$ in the above formula being replaced by a quadratic and cubic one respectively. One may also obtain similar formulae to those of the latter theorem for the mixed third moment twisted by the square modulus of a Dirichlet polynomial by arguments within the scope of the methods utilised in the present work and with the aid of modern technology of the strength of that of Bettin et al. \cite{Bett} in the interest of surmounting certain technical difficulties.

The exposition of ideas is organised as follows. We begin our journey by presenting some preliminary lemmata in Section \ref{preliminary}. Sections \ref{offdiagonal} and \ref{diagonal} are then devoted to the analysis when $a\geq 2$ of the off-diagonal and diagonal contribution, the case $a=1$ being treated independently in Section \ref{case1}. In Section \ref{simple} we routinarily bound oscillatory integrals involved in the formula at hand and complete the proofs of Theorems \ref{thmabcy} and \ref{prop3k4}. Theorem \ref{ppop} is then discussed and proved in Section \ref{contour}. The memoir concludes in Section \ref{twist} with an analysis and a proof of Theorem \ref{thm10}.

\emph{Notation}: As is customary in number theory, we write $p^{r}|| n$ to denote that $p^{r}| n$ but $p^{r+1}\nmid n.$ Whenever $\varepsilon$ appears in any bound, it will mean that the bound holds for every $\varepsilon>0$, though the implicit constant then may depend on $\varepsilon$. We use $\ll$ and $\gg$ to denote Vinogradov's notation, and write $f\asymp g$ whenever $f\ll g$ and $f\gg g$.

\emph{Acknowledgements}: The author's work was initiated during his visit to Purdue University under Trevor Wooley's supervision and finished at KTH while being supported by the G\"oran Gustafsson Foundation. The author would like to thank him for his guidance and useful comments, Lilian Matthiesen for helpful conversations and both Purdue University and KTH for their support and hospitality, and to acknowledge the activities supported by the NSF Grant DMS-1854398.

\section{Preliminary manoeuvres}\label{preliminary}
We start by recalling the following standard result concerning the asymptotic evaluation of the function $\chi(s)$ defined in (\ref{chichi}). 
\begin{lem}\label{lem601}
Let $t>0$. One then has
$$\chi(1/2+it)=\Big(\frac{2\pi}{t}\Big)^{it}e^{it+i\pi/4}\Big(1+O\Big(\frac{1}{t}\Big)\Big),\ \ \ \ \chi(1/2-it)=\Big(\frac{2\pi}{t}\Big)^{-it}e^{-it-i\pi/4}\Big(1+O\Big(\frac{1}{t}\Big)\Big)$$ as $t\to\infty$.
\end{lem}
\begin{proof} The above formulae follow from the equation right after Titchmarsh \cite[(7.4.3)]{Tit} concerning the asymptotic evaluation of $\chi(1-s).$
\end{proof}

We next demonstrate how the problem shall be reduced to that of computing integrals of products of twisted Dirichlet polynomials, but before accomplishing such an endeavour it is desirable to recall for each $0<\theta<1$ the definition of $D_{\theta}(s)$ in (\ref{diri}). We write $D(s)$ to denote $D_{1/2}(s)$ for the sake of simplicity and introduce
\begin{equation*}P(t)=D(1/2+it)+\chi(1/2+it)D(1/2-it)\end{equation*} for $t\in\mathbb{R},$ where $\chi(s)$ was defined in (\ref{chichi}). It shall also be convenient for further use to define for $T>0$ and fixed $\theta\in\mathbb{R}$ and $a\in\mathbb{N}$ the parameters
\begin{equation}\label{QQQQ}T_{1}=T/2\pi,\ \ \ \ \ \ \ Q=\Big(\frac{aT}{2\pi}\Big)^{\theta}.\end{equation} 
\begin{lem}\label{lemita601}
With the above notation, one has when $0<\theta<1$ that
\begin{equation}\label{prips}M_{a,b,c}^{\theta}(T)=\int_{0}^{T}D_{\theta}(1/2+ait)P(-bt)P(-ct)dt+O\big(T^{3/4}\log T\big)\end{equation} and that whenever $0<\theta\leq 1/2$ then \begin{equation*}M_{a,a,b,c}^{\theta}(T)=\int_{0}^{T}\lvert D_{\theta}(1/2+ait)\rvert^{2}P(-bt)P(-ct)dt+O\big(T^{3/4}(\log T)^{5/2}\big).\end{equation*}
\end{lem}
\begin{proof}

We recall (\ref{QQQQ}) and observe by making a distinction between the diagonal and off-diagonal contribution that upon denoting for given pairs $(n_{1},n_{2})\in\mathbb{N}^{2}$ the expression $M_{n_{1},n_{2}}=2\pi a^{-1}\max(n_{1}^{1/\theta},n_{2}^{1/\theta})$ it then transpires that
\begin{align*}\int_{0}^{T}\lvert D_{\theta}(1/2+ait)\rvert^{2}dt&= \sum_{n\leq Q}n^{-1}(T-2\pi a^{-1} n^{1/\theta})+\sum_{\substack{n_{1}\leq Q\\ n_{2}\leq Q}}(n_{1}n_{2})^{-1/2}\int_{M_{n_{1},n_{2}}}^{T}e^{-iat\log (n_{1}/n_{2})}dt
\\
&=\sum_{n\leq Q}n^{-1}(T-2\pi a^{-1}n^{1/\theta})+O\Bigg(\sum_{\substack{n_{1}\leq Q\\ n_{2}\leq Q}}\frac{(n_{1}n_{2})^{-1/2}}{\lvert\log (n_{1}/n_{2})\rvert}\Bigg).\end{align*}
Therefore, an application of Montgomery and Vaughan \cite[Theorem 2]{MonVau2} yields
\begin{equation}\label{rrio}\int_{0}^{T}\lvert D_{\theta}(1/2+ait)\rvert^{2}dt= T\log Q+(\gamma-\theta) T+O\big(TQ^{-1}+Q\big).\end{equation} 

We continue by defining for $n\in\mathbb{Z}$ the function $$\zeta_{n}(t)=\zeta(1/2+nit)$$ and recalling the approximate functional equation (\ref{approxim}) to derive
\begin{equation*}M_{a,b,c}^{\theta}(T)=\int_{0}^{T}D_{\theta}(1/2+ait)P(-bt)P(-ct)dt+E(T),\end{equation*}wherein the error term $E(T)$ in the above line satisfies the estimate$$E(T)\ll T^{1/4}+E_{1}(T)+E_{2}(T),$$ and the terms $E_{1}(T)$ and $E_{2}(T)$ are defined by the relations
$$E_{1}(T)= \int_{1}^{T}t^{-1/2}\lvert D_{\theta}(1/2+ait)\rvert dt$$ and
$$E_{2}(T)=\int_{1}^{T}t^{-1/4}\lvert D_{\theta}(1/2+ait)\rvert\big( \lvert\zeta_{-b}(t)\rvert +\lvert\zeta_{-c}(t)\rvert\big)dt.$$ We use Cauchy's inequality in conjunction with (\ref{rrio}) to obtain
$$E_{1}(T)\ll (\log T)^{1/2}\Big(\int_{0}^{T}\lvert D_{\theta}(1/2+ait)\rvert^{2}\Big)^{1/2}\ll T^{1/2}\log T.$$
Likewise, integration by parts combined with another application of Cauchy's inequality and both (\ref{rrio}) in conjunction with the asymptotic formula for the second moment of the Riemman Zeta function (see for instance Titchmarsh \cite[Theorem 7.3]{Tit}) delivers
\begin{align*}E_{2}(T)\ll & 1+T^{-1/4}\int_{0}^{T}\lvert D_{\theta}(1/2+ait)\rvert\big( \lvert\zeta_{-b}(t)\rvert +\lvert\zeta_{-c}(t)\rvert\big)dt
\\
&+\max_{r\in\{b,c\}}\int_{1}^{T}t^{-5/4}\int_{0}^{t}\lvert D_{\theta}(1/2+ais)\rvert\lvert\zeta(1/2-irs)\rvert dsdt\ll T^{3/4}(\log T).\end{align*} The combination of the above estimates yields (\ref{prips}).

In order to prepare the ground for the computation of the fourth moment of the Dirichlet polynomial it seems desirable to define for $t>0$ and $n\in\mathbb{N}$ the divisor function
\begin{equation}\label{dividivi}d_{t}(n)=\sum_{\substack{m\mid n\\ m,n/m\leq (at/2\pi)^{\theta}}}1.\end{equation} Then in the same vein as above and upon making the distinction into diagonal and off-diagonal contribution it transpires that 
\begin{align}\label{momen4}\int_{0}^{T}\lvert D_{\theta}(1/2+ait)\rvert^{4}dt=&\int_{0}^{T}\Bigg(\sum_{n\leq (at/2\pi)^{2\theta}}\frac{d_{t}(n)^{2}}{n}\Bigg)dt
\\
&+\int_{0}^{T}\sum_{\substack{n_{1},n_{2}\leq (at/2\pi)^{2\theta}\\ n_{1}\neq n_{2}}}\frac{d_{t}(n_{1})d_{t}(n_{2})}{(n_{1}n_{2})^{1/2}}e^{-iat\log (n_{1}/n_{2})}dt.\nonumber
\end{align}
A careful perusal of the second summand in the above equation reveals that by integrating accordingly and deleting the restriction about the size of the divisors in (\ref{dividivi}) one has that
$$\int_{0}^{T}\sum_{\substack{n_{1},n_{2}\leq (at/2\pi)^{2\theta}\\ n_{1}\neq n_{2}}}\frac{d_{t}(n_{1})d_{t}(n_{2})}{(n_{1}n_{2})^{1/2}}e^{-iat\log (n_{1}/n_{2})}dt\ll \sum_{\substack{n_{1},n_{2}\leq Q^{2}\\ n_{1}\neq n_{2}}}\frac{d(n_{1})d(n_{2})}{(n_{1}n_{2})^{1/2}\lvert \log(n_{1}/n_{2})\rvert}.$$ It is then apparent that another application of Montgomery and Vaughan \cite[Theorem 2]{MonVau2} enables one to obtain
\begin{equation}\label{kkkisp}\int_{0}^{T}\sum_{\substack{n_{1},n_{2}\leq (at/2\pi)^{2\theta}\\ n_{1}\neq n_{2}}}\frac{d_{t}(n_{1})d_{t}(n_{2})}{(n_{1}n_{2})^{1/2}}e^{-iat\log (n_{1}/n_{2})}dt\ll \sum_{n\leq Q^{2}}d(n)^{2}\ll Q^{2}(\log T)^{3},\end{equation}
and that integrating trivially one gets
\begin{equation}\label{kkirp}\int_{0}^{T}\Bigg(\sum_{n\leq (at/2\pi)^{2\theta}}\frac{d_{t}(n)^{2}}{n}\Bigg)dt\ll T\sum_{n\leq Q^{2}}\frac{d(n)^{2}}{n}\ll T(\log T)^{4},\end{equation} which combined with the preceding lines delivers 
\begin{equation}\label{priosa}\int_{0}^{T}\lvert D_{\theta}(1/2+ait)\rvert^{4}dt\ll T(\log T)^{4}.\end{equation}

We shall resume our discussion by writing for convenience $$E_{\theta}(T)=\Big\lvert M_{a,a,b,c}^{\theta}(T)-\int_{0}^{T}\lvert D_{\theta}(1/2+ait)\rvert^{2}P(-bt)P(-ct)dt\Big\rvert$$ and observing that the approximate functional equation (\ref{approxim}) in conjunction with Cauchy's inequality and the preceding estimates enables one to deduce
\begin{align*}E_{\theta}(T) \ll \int_{0}^{T}t^{-1/2}\lvert D_{\theta}(1/2+ait)\rvert^{2}dt+\max_{r\in\{b,c\}}\int_{0}^{T}t^{-1/4}\lvert D_{\theta}(1/2+ait)\rvert^{2}\lvert \zeta(1/2-irt)\rvert dt.
\end{align*} 
It transpires after an application of integration by parts and equation (\ref{rrio}) that the first summand in the above equation is $O\big(T^{1/2}(\log T)\big)$. Likewise, it seems worth observing that by Cauchy's inequality in conjunction with both the aforementioned formula for the second moment of the Riemann-zeta function and (\ref{priosa}) then one has
\begin{align*}\int_{0}^{t}\lvert D_{\theta}(1/2+ais)\rvert^{2}\lvert \zeta(1/2-irs)\rvert ds
&\ll \Big(\int_{0}^{t}\lvert D_{\theta}(1/2+ais)\rvert^{4}ds\Big)^{1/2}\Big(\int_{0}^{rt}\lvert \zeta(1/2+is)\rvert^{2} ds\Big)^{1/2}
\\
&\ll t(\log t)^{5/2}.
\end{align*}
Therefore, employing integration by parts combined with the preceding estimates enables one to get
\begin{align*}\lvert E_{\theta}(T)\rvert \ll T^{3/4}(\log T)^{5/2},\end{align*}
as desired.
\end{proof}

The following lemma shall make use of some of the formulae deduced in the above proof and may enable one to shift the parameter $T$ by $CTQ^{-1}$ for some constant $C>0$ at the cost of having a negligible impact in the asymptotic evaluation. One could have utilised the methods of Heath-Brown \cite{Heath} involving the use of estimates for Kloosterman sums to obtain asymptotics for sums of the shape 
$$\sum_{\substack{(at/2\pi)^{\theta}\leq q\leq (at/2\pi)^{2\theta}}}\sum_{\substack{n\leq (at/2\pi)^{2\theta}\\ q\mid n}}d(n).$$ Having achieved such an endeavour and a routinary application of Perron's formula would have permitted one to asymptotically evaluate and give account of lower order terms for the moment (\ref{momen4}), the upper bound deduced therein being sufficient for our current needs.
\begin{lem}\label{priros}
Let $C>0$ be a fixed constant and $0<\theta<1$. Then one has that
$$\int_{T}^{T+CTQ^{-1}}D_{\theta}(1/2+ait)\zeta(1/2-bit)\zeta(1/2-cit)dt\ll TQ^{-1/2}+(TQ)^{1/2}.$$ Likewise, it follows that
\begin{align*}\int_{T}^{T+CTQ^{-1}}\lvert D_{\theta}(1/2+ait)\rvert^{2}\zeta(1/2-bit)\zeta(1/2-cit)dt\ll TQ^{-1/2}+T^{1/2}Q.
\end{align*}
\end{lem}
\begin{proof}
We shall start with a routinary application of Holder's inequality to obtain
$$\int_{T}^{T+CTQ^{-1}}D_{\theta}(1/2+ait)\zeta(1/2-bit)\zeta(1/2-cit)dt\ll M_{C,2,\theta}(T)^{1/2}M_{C}(T)^{1/2},$$
wherein the above terms are defined by means of the expressions
$$M_{C,2,\theta}(T)=\int_{T}^{T+CTQ^{-1}}\lvert D_{\theta}(1/2+ait)\rvert^{2}dt,\ \ \ M_{C}(T)=\max_{r\in\{b,c\}}\int_{T}^{T+CTQ^{-1}}\lvert \zeta(1/2+rit)\rvert^{4}dt.$$
It is then apparent that equation (\ref{rrio}) then yields
$$M_{C,2,\theta}(T)\ll TQ^{-1}(\log Q)+Q.$$Likewise, the asymptotic formula for the fourth moment \cite{Heath} allows one to obtain
\begin{equation*}M_{C}(T)\ll T^{7/8+\varepsilon}+TQ^{-1}(\log Q)^{4},\end{equation*} whence combining the preceding equations delivers \begin{align*}\int_{T}^{T+CTQ^{-1}}D_{\theta}(1/2+ait)\zeta(1/2-bit)\zeta(1/2-cit)dt \ll & TQ^{-1}(\log Q)^{5/2}+T^{1/2}(\log Q)^{2}
\\
&+T^{15/16+\varepsilon}Q^{-1/2}+T^{7/16+\varepsilon}Q^{1/2},
\end{align*} which then yields the desired result. 

In a similar fashion, another application of Holder's inequality enables one to deduce 
$$\int_{T}^{T+CTQ^{-1}}\lvert D_{\theta}(1/2+ait)\rvert^{2}\zeta(1/2-bit)\zeta(1/2-cit)dt\ll M_{C,4,\theta}(T)^{1/2}M_{C}(T)^{1/2},$$ wherein $M_{C,4,\theta}(T)$ is defined as
$$M_{C,4,\theta}(T)=\int_{T}^{T+CTQ^{-1}}\lvert D_{\theta}(1/2+ait)\rvert^{4}dt.$$ An insightful perusal of the discussion that leads to (\ref{kkkisp}) then reveals that one may obtain in an analogous manner
$$\int_{T}^{T+CTQ^{-1}}\sum_{\substack{n_{1},n_{2}\leq (at/2\pi)^{2\theta}\\ n_{1}\neq n_{2}}}\frac{d_{t}(n_{1})d_{t}(n_{2})}{(n_{1}n_{2})^{1/2}}e^{-iat\log (n_{1}/n_{2})}dt\ll \sum_{n\ll Q^{2}}d(n)^{2}\ll Q^{2}(\log T)^{3}.$$ Likewise, the same argument that yields (\ref{kkirp}) permits one to deduce that
$$\int_{T}^{T+CTQ^{-1}}\Bigg(\sum_{n\leq (at/2\pi)^{2\theta}}\frac{d_{t}(n)^{2}}{n}\Bigg)dt\ll TQ^{-1}\sum_{n\leq Q^{2}}\frac{d(n)^{2}}{n}\ll TQ^{-1}(\log T)^{4},$$ which combined with the preceding discussion delivers 
\begin{align*}\int_{T}^{T+CTQ^{-1}}\lvert D_{\theta}(1/2+ait)\rvert^{2}\zeta(1/2-bit)\zeta(1/2-cit)dt\ll & TQ^{-1}(\log T)^{4}+(TQ)^{1/2}(\log T)^{7/2}
\\
&+T^{15/16+\varepsilon}Q^{-1/2}+T^{7/16+\varepsilon}Q.
\end{align*} Consequently, the above estimate yields the required conclusion.
\end{proof}

We then apply Lemma \ref{lemita601} to express the corresponding integrals as \begin{equation}\label{piresiti2}M_{a,b,c}^{\theta}(T)=\sum_{j=1}^{3}I_{j,\theta}(T) +O\big(T^{3/4}\log T\big)\end{equation} and \begin{equation}\label{piresiti23}M_{a,a,b,c}^{\theta}(T)=\sum_{j=1}^{3}I_{j,\theta,a}(T) +O\big(T^{3/4}(\log T)^{5/2}\big),\end{equation} the terms $I_{j,\theta}(T)$ and $I_{j,\theta,a}(T)$ being defined at a later point in the memoir.

\section{Off-diagonal contribution}\label{offdiagonal}
We shall devote the present section to recover the relevant computations pertaining to the analysis of the off-diagonal contribution in \cite{Pli}, it being required to introduce first some notation. We fix a triple $(a,b,c)\in\mathbb{R}^{+}$, write $\bfn$ to denote $(n_{1},n_{2},n_{3})\in\mathbb{N}^{3}$ and consider the weighted variables \begin{equation}\label{n1prima}n_{1,\theta}'=n_{1}/a^{\theta},\ \ \ \ n_{2}'=n_{2}/\sqrt{b},\ \ \ \ n_{3}'=n_{3}/\sqrt{c},\end{equation} and the parameters \begin{equation}\label{NNN}N_{\bfn,\theta}=2\pi\max(n_{1,\theta}'^{1/\theta},n_{2}'^{2},n_{3}'^{2}) \ \ \ \ \ \text{and}\ \ \ \ \ \  P_{\bfn}=n_{1}n_{2}n_{3}.\end{equation} We recall (\ref{QQQQ}) and foreshadow the convenience of introducing the set of triples
\begin{equation}\label{Babc}\mathcal{B}_{a,b,c,\theta}=\Big\{\bfn\in\mathbb{N}^{3}:\ \ n_{1}\leq Q,\ \ n_{2}\leq \sqrt{b T_{1}},\ \ n_{3}\leq \sqrt{c T_{1}}\Big\},\end{equation} the condition $\bfn \in \mathcal{B}_{a,b,c,\theta}$ being equivalent to the inequality
\begin{equation}\label{lades}N_{\bfn,\theta}\leq T.\end{equation} We then present to the reader the integral 
\begin{align}\label{Is1}I_{1,\theta}(T)&
=\int_{0}^{T}D_{\theta}(1/2+ait)D(1/2-bit)D(1/2-cit)dt=\sum_{\bfn\in \mathcal{B}_{a,b,c,\theta}}P_{\bfn}^{-1/2}\int_{N_{\bfn}}^{T}\Big(\frac{n_{2}^{b}n_{3}^{c}}{n_{1}^{a}}\Big)^{it}dt.\end{align} We make a distinction between the diagonal contribution and the off-diagonal one to obtain
\begin{equation}\label{I131}I_{1,\theta}(T)=J_{1}^{\theta}(T)+J_{2}^{\theta}(T),\end{equation}where in the above equation one has \begin{equation}\label{0.121}J_{1}^{\theta}(T)=\sum_{\substack{\bfn\in \mathcal{B}_{a,b,c,\theta}\\ n_{1}^{a}=n_{2}^{b}n_{3}^{c}}}(T-N_{\bfn,\theta})P_{\bfn}^{-1/2},\ \ \ \ \ \ \ \ \ J_{2}^{\theta}(T)=\sum_{\substack{\bfn\in \mathcal{B}_{a,b,c,\theta}\\ n_{1}^{a}\neq n_{2}^{b}n_{3}^{c}}}P_{\bfn}^{-1/2}\int_{N_{\bfn}}^{T}\Big(\frac{n_{2}^{b}n_{3}^{c}}{n_{1}^{a}}\Big)^{it}dt.\end{equation}

\begin{lem}\label{prop478}
If $a,b,c\in\mathbb{N}$ with $a<c\leq b$ then for $0<\theta<1$ the bound
$$J_{2}^{\theta}(T)\ll T^{1/2}Q^{1/2}+T^{1+1/2c-1/2a}$$ holds unconditionally. If one assumes Conjecture \ref{conj11} then one further has
$$J_{2}^{\theta}(T)\ll T^{1/2}Q^{1/2}+Q^{1/2+3a/2c+\varepsilon}.$$
\end{lem}

\begin{proof}
It follows from \cite[Lemma 3.2]{Pli}.
\end{proof}

\section{Diagonal contribution analysis}\label{diagonal}
Equipped with the above estimates we are prepared to obtain an asymptotic formula for $I_{1,\theta}(T)$ defined in (\ref{Is1}) in the context underlying Theorem \ref{thmabcy}. We recall the definitions in (\ref{n1prima}) and the subsequent lines and write whenever $a<c\leq b$ the equation
\begin{equation}\label{J1J9}J_{1}^{\theta}(T)=\sigma_{a,b,c}T-J_{3}^{\theta}(T)-J_{4}^{\theta}(T),\end{equation}
wherein the above terms are \begin{equation}\label{0.1212}J_{3}^{\theta}(T)=T\sum_{\substack{n_{1}^{a}=n_{2}^{b}n_{3}^{c}\\ \max(n_{1,\theta}'^{1/\theta},n_{2}'^{2},n_{3}'^{2})>T_{1}}}P_{\bfn}^{-1/2},\ \ \ \ \ \ \ J_{4}^{\theta}(T)=\sum_{\substack{\bfn\in \mathcal{B}_{a,b,c,\theta}\\ n_{1}^{a}=n_{2}^{b}n_{3}^{c}}}N_{\bfn,\theta}P_{\bfn}^{-1/2},\end{equation}
and wherein the constant $\sigma_{a,b,c}$ is defined by means of the formula
\begin{equation}\label{sigmaa}\sigma_{a,b,c}=\sum_{n_{1}^{a}=n_{2}^{b}n_{3}^{c}}P_{\bfn}^{-1/2},\end{equation}
the convergence of the above series stemming from the inequality among the coefficients.
\begin{prop}\label{lem602z} 
Let $a,b,c\in\mathbb{N}$ satisfying $(a,b,c)=1$ with $a<c\leq b$ and let $0<\theta<\min(c/2a,1)$. If both $a\nmid b$ and $a\nmid c$ then one has
$$J_{1}^{\theta}(T)=\sigma_{a,b,c}T+O\big(TQ^{-1/2}(\log T)^{2}\big).$$ Therefore, the asymptotic formula
$$I_{1,\theta}(T)=\sigma_{a,b,c}T+O(Q^{1/2+3a/2c+\varepsilon}+TQ^{-1/2}(\log T)^{2}+T^{1/2}Q^{1/2})$$ holds under the assumption of Conjecture \ref{conj11}.
\end{prop}
\begin{proof}
A straight substitution of the identity $n_{1}=n_{2}^{b/a}n_{3}^{c/a}$ into the sums pertaining to the definitions in (\ref{0.1212}) already delivers error terms of the shape $O(TQ^{-1/2+a/2c})$. Such an approach though does not further exploit the property of $n_{2}^{b}n_{3}^{c}$ being a perfect $a$-th power. To remedy this situation and obtain an estimate superior to that stemming from the aforementioned shortcut, parametrizing the equation at hand shall be a desideratum. For this purpose we then introduce \begin{equation}\label{a1a2a3}a_{2}=a/(a,b),\ \ \ \ b_{2}=b/(a,b),\end{equation} $$a_{3}=a/(a,c),\ \ \ \ c_{3}=c/(a,c),\ \ \ \ A=\frac{a}{(a,b)(a,c)},$$ 
wherein the property of the parameter $A$ being an integer stems from the coprimality condition $(a,b,c)=1$. We begin by noting in view of the divisibility conditions at the statement of the proposition that $a_{2}\neq 1\neq a_{3}.$ The main idea underlying the parametrization will have its reliance on the classification of the divisors $d | n_{2}$ (and similarly for $n_{3}$) according to whether the number $d^{b}$ shall or not be a perfect $a$-th power, such a condition amounting to the number $d$ being a perfect $a_{2}$-th power. If instead $d$ does not satisfy such a property, some divisibility relation between both $d$ and $n_{3}$ should then hold, and exploiting such a dependence among the divisors of both $n_{2}$ and $n_{3}$ shall ultimately lead to sharper conclusions. In order to put these ideas into effect we write
$$n_{2}=r_{2}^{a_{2}}\prod_{u=1}^{a_{2}-1}d_{u}^{u},\ \ \ \ \ \ n_{3}=r_{3}^{a_{3}}\prod_{v=1}^{a_{3}-1}f_{v}^{v},$$ where $d_{u}$ and $f_{v}$ denote squarefree numbers with the property that $(d_{u_{1}},d_{u_{2}})=1$ whenever $u_{1}\neq u_{2}$ and similarly for $f_{v}$. In view of the above definitions, we find it convenient to note that $a\nmid ub$ whenever $1\leq u\leq a_{2}-1.$ We also observe that for fixed $u$ such that $(a,c)\mid u$ and satisfying the above line of inequalities, there is a unique solution to the congruence
$$ub+\alpha_{u}c\equiv 0 \mmod{a},\ \ \  \ \ \ 1\leq \alpha_{u}\leq a_{3}-1.$$ 

By the preceding discussion it transpires that then there is some $1\leq v\leq a_{3}-1$ satisfying $f_{v}=d_{u}$ with $v=\alpha_{u}$ as above, a concomitant aspect of the coprimality condition $(a,b,c)=1$ entailing the necessity of the provisos $(a,c)| u$ and $(a,b)| \alpha_{u}.$ Likewise, an analogous argument may be employed to deduce, for $v$, the existence of some $1\leq u\leq a_{2}-1$ with the property that $d_{u}=f_{v}$. Therefore, one may parametrize the triples satisfying the equation by 
\begin{equation*}n_{1}=r_{2}^{b_{2}}r_{3}^{c_{3}}P_{\bfd},\ \ \ \ \ n_{2}=r_{2}^{a_{2}}\prod_{j=1}^{A-1}d_{j}^{j(a,c)},\ \ \ \ \ n_{3}=r_{3}^{a_{3}}\prod_{j=1}^{A-1}d_{j}^{\alpha_{j}(a,b)},\end{equation*}
where we defined for $\bfd=(d_{1},\ldots,d_{A-1})$ the products \begin{equation*}P_{\bfd}=\prod_{j=1}^{A-1}d_{j}^{(bj(a,c)+c\alpha_{j}(a,b))/a},\ \ \ \ \ \ M_{\bfd}=\prod_{j=1}^{A-1}d_{j}^{j(a,c)+\alpha_{j}(a,b)},\end{equation*} the latter parameter being introduced for prompt use and wherein
$$jb_{2}+\alpha_{j}c_{3}\equiv 0\mmod{A},\ \ \ \ \ \ 1\leq j\leq A-1,\ \ 1\leq \alpha_{j}\leq A-1.$$ Combining the above equations one then finds that
\begin{equation}\label{sigmapa}\sigma_{a,b,c}=\sum_{r_{2},r_{3}} r_{2}^{-(a_{2}+b_{2})/2}r_{3}^{-(a_{3}+c_{3})/2}\sum_{\bfd}P_{\bfd}^{-1/2}M_{\bfd}^{-1/2},\end{equation}
the inner sum running over tuples $\bfd$ comprising square-free pairwise coprime integers. We will first analyse the term $J_{3}^{\theta}(T)$. It might be worth noting by drawing the reader's attention to (\ref{n1prima}) that as a consequence of the inequalities among the coefficients and the underlying equation satisfied by the triple $\bfn$, one has that 
$$n_{1,\theta}'^{1/\theta}=a^{-1}n_{2}^{b/a\theta}n_{3}^{c/a\theta},$$ whence by recalling the assumptions on $\theta$ in the statement then it transpires by (\ref{NNN}) that
\begin{equation}\label{assumth} n_{1,\theta}'^{1/\theta}=\frac{1}{2\pi}N_{\bfn,\theta}=\max(n_{1,\theta}'^{1/\theta},n_{2}'^{2},n_{3}'^{2}).\end{equation} Equipped with such an observation and recalling (\ref{QQQQ}) we plug the above parametrization into (\ref{0.1212}) and sum over $r_{2}$ first to obtain
\begin{align}\label{J3111}J_{3}^{\theta}(T)=&T\sum_{r_{2}^{b_{2}}r_{3}^{c_{3}}P_{\bfd}\geq Q}r_{2}^{-(a_{2}+b_{2})/2}r_{3}^{-(a_{3}+c_{3})/2}P_{\bfd}^{-1/2}M_{\bfd}^{-1/2}=J_{3,1}(T,Q)+J_{3,2}(T,Q),
\end{align}
wherein $J_{3,1}(T,Q)$ denotes the sum over triples such that $r_{3}^{c_{3}}P_{\bfd}< Q$, the term $J_{3,2}(T,Q)$ comprising triples not satisfying the preceding inequality. In order to bound the first summand, which we denote by $J_{3,1}(T,Q)$ for convenience, it might be worth considering first the case when \begin{equation}\label{inver}-a_{3}/2+c_{3}a_{2}/2b_{2}-c_{3}/b_{2}> -1.\end{equation} 

Under such circumstances, it transpires that
\begin{align}\label{pirulet}J_{3,1}(T,Q)&\ll TQ^{-1/2-a_{2}/2b_{2}+1/b_{2}}\sum_{r_{3}^{c_{3}}P_{\bfd}< Q}P_{\bfd}^{-1/b_{2}+a_{2}/2b_{2}}r_{3}^{-a_{3}/2+c_{3}a_{2}/2b_{2}-c_{3}/b_{2}}M_{\bfd}^{-1/2}
\\
&\ll TQ^{-1/2-a_{3}/2c_{3}+1/c_{3}}\sum_{P_{\bfd}\leq Q}P_{\bfd}^{-1/c_{3}+a_{3}/2c_{3}}M_{\bfd}^{-1/2}\nonumber.\end{align} The reader may find it desirable to observe that the exponent of the factor $d_{1}$ involved in the term $M_{\bfd}^{-1/2}$ is at most $-1$, the exponents appertaining to the rest of the factors being at most $-3/2$. We also note that in the above sum then 
$$P_{\bfd}^{-1/c_{3}+a_{3}/2c_{3}}\ll 1+Q^{-1/c_{3}+a_{3}/2c_{3}}.$$ We observe that the previous estimate encompasses the cases when the exponent on the left side is both non-negative and negative. The preceding discussion then yields
\begin{equation}\label{j31}J_{3,1}(T,Q)\ll TQ^{-1/2-a/2c+1/c_{3}}(\log Q)+TQ^{-1/2}(\log T)\ll TQ^{-1/2}(\log Q),\end{equation} the last step stemming from the assumption $a\nmid c$, it in turn entailing the inequality $$\frac{1}{c_{3}}=\frac{(a,c)}{c}\leq \frac{a}{2c}.$$ 

If, on the contrary, condition (\ref{inver}) does not hold then
\begin{equation*}J_{3,1}(T,Q)\ll TQ^{-1/2-a_{2}/2b_{2}+1/b_{2}}(\log Q)\sum_{P_{\bfd}\leq Q}P_{\bfd}^{-1/b_{2}+a_{2}/2b_{2}}M_{\bfd}^{-1/2},\end{equation*} and an argument reminiscent of the above yields \begin{equation}\label{ec2424}J_{3,1}(T,Q)\ll TQ^{-1/2}(\log Q)^{2}+TQ^{-1/2-a/2b+1/b_{2}}(\log Q)^{2}\ll TQ^{-1/2}(\log Q)^{2},\end{equation} the last inequality being a consequence of the bound $b_{2}\geq 2b/a$ stemming from the fact that $a\nmid b.$ It might be pertinent to clarify that the extra factor of $(\log Q)$ has been added to encompass the case when the left side of (\ref{inver}) is $-1$.

Likewise, an analogous argument to the one employed above then reveals that
\begin{align*}J_{3,2}(T,Q) 
 \ll& TQ^{-1/2-a_{3}/2c_{3}+1/c_{3}}\sum_{P_{\bfd}\leq Q}P_{\bfd}^{-1/c_{3}+a_{3}/2c_{3}}M_{\bfd}^{-1/2}+T\sum_{P_{\bfd}\geq Q}P_{\bfd}^{-1/2}M_{\bfd}^{-1/2}.
\end{align*}
It seems appropiate to note that in the second summand of the above equation one has
$P_{\bfd}^{-1/2}\ll Q^{-1/2}.$ Therefore, the above observation in conjunction with the discussion following (\ref{pirulet}) yields the bound
$$J_{3,2}(T,Q)\ll TQ^{-1/2}\log Q,$$ a combination of the preceding estimates thereby delivering
\begin{equation*}\label{J3J3J3}J_{3}^{\theta}(T)\ll  TQ^{-1/2}(\log Q)^{2}.\end{equation*}

We next focus on the term $J_{4}^{\theta}(T)$ in (\ref{0.1212}). We shall use the parametrization employed in the analysis of $J_{3}^{\theta}(T)$ herein as well and observe that as was noted above one necessarily has that the tuples involved in the sum in the definition of $J_{4}^{\theta}(T)$ have the property that $2\pi n_{1,\theta}'^{1/\theta}=N_{\bfn,\theta}.$ Therefore, in view of (\ref{n1prima}) and (\ref{lades}) and recalling (\ref{a1a2a3}) it transpires that
\begin{align*}J_{4}^{\theta}(T)\ll &
\sum_{r_{2}^{b_{2}}r_{3}^{c_{3}}P_{\bfd}\leq Q}r_{2}^{b_{2}/\theta-(b_{2}+a_{2})/2}r_{3}^{c_{3}/\theta-(c_{3}+a_{3})/2}P_{\bfd}^{1/\theta-1/2}M_{\bfd}^{-1/2}.
\end{align*}
Recalling (\ref{QQQQ}) and summing over $r_{2}$ first we find that 
\begin{align*}J_{4}^{\theta}(T)&\ll TQ^{-1/2-a_{2}/2b_{2}+1/b_{2}}\sum_{r_{3}^{c_{3}}P_{\bfd}\leq Q}r_{3}^{a_{2}c_{3}/2b_{2}-c_{3}/b_{2}-a_{3}/2}P_{\bfd}^{a_{2}/2b_{2}-1/b_{2}}M_{\bfd}^{-1/2}
,\end{align*} where the reader may find it useful to recall that the right side of the above line appeared in (\ref{pirulet}). Consequently, equation (\ref{j31}) in conjunction with (\ref{ec2424}) and the above line of inequalities yields \begin{equation*}\label{domi}J_{4}^{\theta}(T)\ll TQ^{-1/2}(\log Q)^{2}.\end{equation*} The preceding discussion combined with Lemma \ref{prop478} and  (\ref{I131}) and (\ref{J1J9}) delivers the proof.
\end{proof}
We shift our attention to the perusal of the instance when either $a\mid b$ or $a\mid c$ to the end of deriving an analogous formula comprising lower order terms. For such purposes, it seems convenient to introduce for fixed $\theta$ the parameter \begin{equation}\label{sthe}\lambda_{\theta}=\frac{2}{\theta}-1,\end{equation} and define for tuples $(a,r,s)\in\mathbb{N}^{3}$ not of the shape $(1,r,r)$ or $(r,r,s)$ the  constants \begin{equation}\label{Cars}C_{a,r,s}=\frac{4ra}{\theta(r-a)(\lambda_{\theta}r+a)}\zeta(as/2r+a/2)(2\pi/a)^{\theta(1/2-a/2r)}.\end{equation} 
\begin{prop}\label{propAA}
Let $a,b,c\in\mathbb{N}$ such that $a<c\leq b$ and $a\geq 2$ with either $a\mid b$ or $a\mid c$ and satisfying $(a,b,c)=1$. Let $\theta>0$ for which $\theta\leq\min(c/2a,1\big).$ Then if $\{r,s\}=\{b,c\}$ and such that $a\mid r$ and $a\nmid s$ one has
$$J_{1}^{\theta}(T)=K_{a,b,c}T-C_{a,r,s}T^{1-\theta(1/2-a/2r)}+O(T^{1-\theta/2}),$$ wherein \begin{equation}\label{Kabc}K_{a,b,c}=\zeta\big((a_{2}+b_{2})/2\big)\zeta\big((a_{3}+c_{3})/2\big).\end{equation} Under the assumption of Conjecture \ref{conj11} then it transpires that
$$I_{1,\theta}(T)=K_{a,b,c}T-C_{a,r,s}T^{1-\theta(1/2-a/2r)}+O(T^{1-\theta/2}+T^{1/2+\theta/2}+T^{\theta/2+3\theta a/2c+\varepsilon}).$$
\end{prop}
Before starting the proof, it seems appropiate to clarify that in view of the coprimality condition among the coefficients, the situations $a\mid b$ and $a\mid c$ are mutually exclusive.
\begin{proof}

Upon recalling (\ref{a1a2a3}), it transpires that the parametrization of the corresponding underlying equation in (\ref{0.1212}) would then be \begin{equation*}n_{1}=r_{2}^{b_{2}}r_{3}^{c_{3}},\ \ \ \ \ \ \ \ n_{2}=r_{2}^{a_{2}},\ \ \ \ \ \ \ \ n_{3}=r_{3}^{a_{3}},\end{equation*}whence on recalling (\ref{sigmapa}) it is apparent in this context that $\sigma_{a,b,c}=K_{a,b,c}$. We begin by examining first the case $a\mid b$ and $a\nmid c$, the aforementioned proviso in conjunction with the condition $(a,b,c)=1$ entailing $$a_{2}=1,\ \ \ b_{2}=b/a,\ \ \ \ a_{3}=a,\ \ \ \ c_{3}=c,\ \ \ \ \ (a,c)=1.$$ Moreover, in view of the restrictions on $\theta$ in the statement of the proposition then a routinary argument in the same vein as in equation (\ref{assumth}) yields $$J_{3}^{\theta}(T)=T\sum_{r_{2}^{b/a}r_{3}^{c}\geq Q}r_{2}^{-(a+b)/2a}r_{3}^{-(a+c)/2}.$$ 

It seems worth noting that on recalling the definition of $J_{3,1}(T,Q)$ in (\ref{J3111}) one has in this context
\begin{align*}J_{3,1}(T,Q)=&\frac{2a}{b-a}TQ^{-1/2+a/2b}\sum_{r_{3}^{c}< Q}r_{3}^{-a/2-ac/2b}+O\Big(TQ^{-1/2-a/2b}\sum_{r_{3}^{c}< Q}r_{3}^{ac/2b-a/2}\Big),
\end{align*}whence regardless of the size of the exponent cognate to $r_{3}$ in the above error term one has 
$$TQ^{-1/2-a/2b}\sum_{r_{3}^{c}< Q}r_{3}^{ac/2b-a/2}\ll TQ^{-1/2-a/2b}(\log T)+TQ^{-1/2-a/2c+1/c}\ll TQ^{-1/2}.$$
Moreover, we remark that the exponent of $r_{3}$ in the main term of the above expression for $J_{3,1}(T,Q)$ is smaller than $-1$, whence summing over $r_{3}$ yields
$$J_{3,1}(T,Q)=\frac{2a}{b-a}\zeta(ac/2b+a/2)TQ^{-1/2+a/2b}+O(TQ^{-1/2}).$$Likewise, we remind the reader of the definition of $J_{3,2}(T,Q)$ in (\ref{J3111}) and observe that under the above assumptions one has that
$$J_{3,2}(T,Q)\ll T\sum_{r_{3}\geq Q^{1/c}}r_{3}^{-(a+c)/2}\ll TQ^{-1/2-a/2c+1/c}\ll TQ^{-1/2},$$wherein the last inequality we utilised the fact that $a\geq 2$. Combining the above equations then delivers
\begin{equation}\label{ecuJ3}J_{3}^{\theta}(T)=\frac{2a}{b-a}\zeta(ac/2b+a/2)TQ^{-1/2+a/2b}+O(TQ^{-1/2}).\end{equation}

We recall (\ref{QQQQ}), (\ref{lades}), (\ref{0.1212}) and (\ref{assumth}) to the end of noting for the the perusal of $J_{4}^{\theta}(T)$ and in the same vein as therein that as a consequence of the restriction on $\theta$ in the statement of the proposition then one has 
\begin{align*}J_{4}^{\theta}(T)&=\frac{2\pi}{a}\sum_{r_{2}^{b/a}r_{3}^{c}\leq Q}r_{2}^{b/a\theta-(b+a)/2a}r_{3}^{c/\theta-(a+c)/2}
\\
&=\frac{2a}{\lambda_{\theta}b+a}TQ^{-1/2+a/2b}\sum_{r_{3}^{c}\leq Q}r_{3}^{-ac/2b-a/2}+O\Big(TQ^{-1/2-a/2b}\sum_{r_{3}^{c}\leq Q}r_{3}^{ac/2b-a/2}\Big).\nonumber
\end{align*}
It seems pertinent to remark that in view of the condition $a\geq 2$ earlier assumed, it transpires that the above error term is $O(TQ^{-1/2}).$ Moreover, the same proviso assures that the exponent cognate to $r_{3}$ in the above main term is smaller than $-1$, whence the preceding discussion in conjunction with the above equation delivers
$$J_{4}^{\theta}(T)=\frac{2a\zeta\big(ac/2b+a/2\big)}{\lambda_{\theta}b+a}TQ^{-1/2+a/2b}+O(TQ^{-1/2}),$$ which combined with (\ref{ecuJ3}), Lemma \ref{prop478} and  (\ref{I131}) and (\ref{J1J9}) yields the desired result whenever $a\mid b$. Likewise, the same argument for the instance $a\mid c$ delivers an analogous formula with $b$ replacing $c$. The previous observation then completes the proof of the lemma.
\end{proof}

\section{Two lower order terms stemming from the diagonal contribution}\label{case1}
The discussion herein shall be devoted to the case $a=1$, an intrincate analysis involving more sophisticated techniques being required to the end of deriving a formula comprising two secondary terms, it being profitable presenting first a technical lemma that shall be employed on further occasions. For such purposes we introduce the meromorphic functions
\begin{equation}\label{rris}G_{b,c}(s)=\zeta\Big(bs+\frac{1+b}{2}\Big)\zeta\Big(cs+\frac{1+c}{2}\Big),\ \ \ \ \ \ \ \ \ \ \ \ \ G_{1,b,c}(s)=\zeta(s+1)G_{b,c}(s).\end{equation}
We also recall (\ref{QQQQ}) and define the integrals along the critical line
\begin{equation}\label{Mbc}L_{b,c}(Q)=\int_{-1/2-iQ^{2}}^{-1/2+iQ^{2}}G_{b,c}(s)Q^{s}\frac{ds}{s},\ \ \ \ \ \ L_{1,b,c}(Q)=\int_{-1/2-iQ^{2}}^{-1/2+iQ^{2}}G_{1,b,c}(s)Q^{s}\frac{ds}{s}.\end{equation}

\begin{lem}\label{lem6.1}
Let $b,c\in\mathbb{N}$ such that $1\leq c\leq b$. One has the estimates
$$L_{b,c}(Q)\ll Q^{-1/2}(\log Q)^{2},\ \ \ \ \ \ \ L_{1,b,c}(Q)\ll Q^{-1/2}(\log Q)^{13/4},$$ and on denoting $\xi_{m}=(-1)^{m}$ then for any $\delta>-1/2$ it follows that
\begin{align*}\int_{-1/2+i\xi_{m}Q^{2}}^{\delta+i\xi_{m}Q^{2}}G_{b,c}(s)Q^{s}\frac{ds}{s}\ll& Q^{\delta-4/3},\ \ \ \ \ \ \int_{-1/2+i\xi_{m}Q^{2}}^{\delta+i\xi_{m}Q^{2}}G_{1,b,c}(s)Q^{s}\frac{ds}{s}\ll Q^{\delta-1},\ \ \ \ \ m=1,2.
\end{align*}  
\end{lem}
\begin{proof}
We begin the proof by analysing the first integral and observing that
$$L_{b,c}(Q)=Q^{-1/2}\int_{-Q^{2}}^{Q^{2}}\zeta(1/2+bit)\zeta(1/2+cit)Q^{it}\frac{dt}{(-1/2+it)},$$whence an application of Holder's inequality yields
$$L_{b,c}(Q)\ll Q^{-1/2}\Big(1+\int_{1}^{Q^{2}}\lvert \zeta(1/2+it)\lvert^{2}\frac{dt}{t}\Big).$$ It then transpires that a combination of the asymptotic formula for the second moment due to Hardy-Littlewood (see Titchmarsh \cite[Theorem 7.3]{Tit}) in conjunction with integration by parts and the above inequality yields
$$L_{b,c}(Q)\ll Q^{-1/2}\Bigg(\log Q+\int_{1}^{Q^{2}}\frac{\log t}{t}dt\Bigg)\ll Q^{-1/2}(\log Q)^{2},$$ as desired. Likewise, following an analogous procedure and upon recalling the estimate $$\int_{0}^{T}\lvert \zeta(1/2+it)\rvert^{3}dt\ll T(\log T)^{9/4}$$ derived in \cite{Bet2} enables one to deduce
\begin{align*}L_{1,b,c}(Q)\ll Q^{-1/2}\Big(1+\int_{1}^{Q^{2}}\lvert \zeta(1/2+it)\lvert^{3}\frac{dt}{t}\Big)&\ll Q^{-1/2}(\log Q)^{13/4}.
\end{align*}
For the second estimate we begin by recalling the bound
\begin{equation*}\label{Bourg}\zeta(1/2+it)\ll t^{1/6-\tau}\end{equation*} for some explicit $\tau>0$ (see Titchmarsh \cite[Theorem 5.18]{Tit}), an immediate consequence of which being on the horizontal line $\text{Im}(s)=\xi_{m}Q^{2}$ at hand that
$$G_{b,c}(s)\ll Q^{2/3}, \ \ \ \ \ \ G_{1,b,c}(s)\ll Q.$$ The preceding observation thus yields
$$\int_{-1/2+i\xi_{m}Q^{2}}^{\delta+i\xi_{m}Q^{2}}G_{b,c}(s)Q^{s}\frac{ds}{s}\ll Q^{-4/3+\delta},\ \ \ \int_{-1/2+i\xi_{m}Q^{2}}^{\delta+i\xi_{m}Q^{2}}G_{1,b,c}(s)Q^{s}\frac{ds}{s}\ll Q^{-1+\delta},\ \ \ \  m=1,2,$$
which completes the proof.

\end{proof}

We next present a technical lemma estimating the error term emerging from the application of Perron's formula that shall be utilised on several ocassions. An insightful inspection of the statement of Lemma \ref{priros} reveals that one may translate $T$ to $T+CTQ^{-1}$ for any constant $C>0$ if necessary, the error terms stemming from such a manoeuvre being $O\big(TQ^{-1/2}+(TQ)^{1/2}\big)$ in the setting underlying Theorems \ref{prop3k4} and \ref{ppop} and $O\big(TQ^{-1/2}+T^{1/2}Q\big)$ in the context of Theorem \ref{thm10}. In view of the above considerations it is then apparent that one may assume 
\begin{equation}\label{fracc}
\|Q\|\gg 1,
\end{equation}
such a proviso alleviating some tedious computations that would have otherwise been required in the analysis of the error term arising from the aforementioned use of Perron's formula. It then seems worth defining the sums
\begin{equation}\label{pries}R_{1}(T,Q)= T\sum_{Q/2<n_{2}^{b}n_{3}^{c}<2Q}n_{2}^{-(1+b)/2}n_{3}^{-(1+c)/2}\min\big(1,Q^{-2}\big\rvert\log\big(n_{2}^{b}n_{3}^{c}Q^{-1}\big)\big\lvert^{-1}\big)\end{equation} and 
\begin{equation}\label{priest}R_{1,1}(T,Q)= T\sum_{Q/2<n_{1}n_{2}^{b}n_{3}^{c}<2Q}n_{1}^{-1}n_{2}^{-(1+b)/2}n_{3}^{-(1+c)/2}\min\big(1,Q^{-2}\big\rvert\log\big(n_{1}n_{2}^{b}n_{3}^{c}Q^{-1}\big)\big\lvert^{-1}\big).\end{equation}

\begin{lem}\label{lem6.2}
Assume that $\|Q\|\gg 1$. Then with the above notation, for $b,c\in\mathbb{N}$ satisfying $b\geq c$ one has that
$$R_{1}(T,Q)\ll TQ^{-1}(\log Q),\ \ \ \ \ \ \ \ \ \ R_{1,1}(T,Q)\ll TQ^{-1}(\log Q)^{2}.$$
\end{lem}

\begin{proof}
We shall focus first on the perusal of $R_{1}(T,Q)$. In view of the assumption on $Q$ it follows that
$$\lvert \log\big(n_{2}^{b}n_{3}^{c}Q^{-1}\big)\rvert\asymp \frac{\lvert n_{2}^{b}n_{3}^{c}-Q\rvert}{Q}\gg Q^{-1}.$$
Inserting the above line in (\ref{pries}) enables one to derive
$$R_{1}(T,Q)\ll TQ^{-1}\sum_{\substack{Q/2<n_{2}^{b}n_{3}^{c}<2Q}} n_{2}^{-(1+b)/2}n_{3}^{-(1+c)/2}\ll TQ^{-3/2}\sum_{\substack{Q/2<n_{2}^{b}n_{3}^{c}<2Q}} n_{2}^{-1/2}n_{3}^{-1/2}.$$ 
We sum over $n_{3}$ in the above equation to the end of obtaining
$$R_{1}(T,Q)\ll TQ^{-3/2+1/2c}\sum_{\substack{n_{2}^{b}<2Q}} n_{2}^{-1/2-b/2c}\ll TQ^{-1}(\log Q),$$ from where the required estimate stated above follows. We shall proceed in a similar fashion for bounding $R_{1,1}(Q,T)$ and begin by noting as was earlier done that $$\lvert \log\big(n_{1}n_{2}^{b}n_{3}^{c}Q^{-1}\big)\rvert\asymp \frac{\lvert n_{1}n_{2}^{b}n_{3}^{c}-Q\rvert}{Q}\gg Q^{-1}.$$ Inserting the above equation in (\ref{priest}) allows one to obtain
\begin{align*}R_{1,1}(Q,T)&\ll TQ^{-1}\sum_{\substack{Q/2<n_{1}n_{2}^{b}n_{3}^{c}<2Q}} n_{1}^{-1}n_{2}^{-(1+b)/2}n_{3}^{-(1+c)/2}
\\
&\ll TQ^{-1}\sum_{\substack{n_{2}^{b}n_{3}^{c}<2Q}} n_{2}^{-(1+b)/2}n_{3}^{-(1+c)/2}\ll TQ^{-1}(\log Q)^{2},
\end{align*} as desired.
\end{proof}

In order to make progress it seems pertinent to employ the technical lemmata presented above to the end of computing the term $J_{4}^{\theta}(T)$ first, it being defined in (\ref{0.1212}). For such purposes we recall (\ref{sthe}), consider for pairs $(r,s)\in\mathbb{N}^{2}$ the constant \begin{equation}\label{Brs}D_{r,s}=\frac{2\zeta(s/2r+1/2)}{\lambda_{\theta}r+1},\end{equation} and for $r\in\mathbb{N}$ we introduce
\begin{equation}\label{F0r}F_{0,r}=\frac{4\gamma}{\lambda_{\theta}r+1}-\frac{4}{(\lambda_{\theta}r+1)^{2}},\ \ \ \ \ \ F_{1,r}=\frac{2}{r(\lambda_{\theta}r+1)},\end{equation} wherein $\gamma$ denotes the Euler's constant.
\begin{lem}\label{lemJ4}
Let $a=1$ and $b,c\in\mathbb{N}$ satisfying $b>c\geq 1$. Then assuming (\ref{fracc}) one has
$$J_{4}^{\theta}(T)=D_{c,b}TQ^{-1/2+1/2c}+D_{b,c}TQ^{-1/2+1/2b}+O\big(TQ^{-1/2}(\log T)^{2}\big).$$
If on the contrary $b=c$ then $$J_{4}^{\theta}(T)=TQ^{-1/2+1/2b}(F_{1,b}(\log Q)+F_{0,b})+O\big(TQ^{-1/2}(\log T)^{2}\big).$$
\end{lem}

\begin{proof}

We define for convenience the meromorphic function 
$$H_{b,c}(s)=\zeta\big(bs+(1+b)/2-b/\theta\big)\zeta\big(cs+(1+c)/2-c/\theta\big)$$ and take $\delta_{Q}=(\log Q)^{-1}$. We then recall (\ref{QQQQ}) and equation (\ref{sthe}) and observe that an application of Perron's formula (see \cite[Theorem 5.2]{MonVau}) and the same argument as in (\ref{assumth}) in conjunction with (\ref{lades}) yields
\begin{align*}J_{4}^{\theta}(T)=&2\pi \sum_{r_{2}^{b}r_{3}^{c}\leq Q}r_{2}^{b/\theta-(b+1)/2}r_{3}^{c/\theta-(c+1)/2}=i^{-1}\int_{1/\theta+\delta_{Q}-iQ^{2}}^{1/\theta+\delta_{Q}+iQ^{2}}H_{b,c}(s)Q^{s}\frac{ds}{s}+E(T,Q),
\end{align*}
wherein $E(T,Q)$ satisfies the estimate
$$E(T,Q)\ll R_{1}(T,Q)+TQ^{-2},$$ the term $R_{1}(T,Q)$ having been previously defined in (\ref{pries}). It seems worth momentarily pausing our discussion and remark that the error term in equation (5.8) of Montgomery and Vaughan \cite[Theorem 5.2]{MonVau} is $O(TQ^{-2})$ in the present context. We focus first on the instance $c<b$, recall (\ref{sthe}) and shift as is customary the line of integration to $\text{Re}(s)=\lambda_{\theta}/2$ in the above integral to obtain
\begin{align*}J_{4}^{\theta}(T)=&D_{c,b}TQ^{-1/2+1/2c}+D_{b,c}TQ^{-1/2+1/2b}+E(T,Q)+2\pi E_{1}(T,Q)
\\
&+2\pi\sum_{m=1}^{2}\xi_{m}E_{2,m}(T,Q),
\end{align*}
wherein we wrote $$E_{1}(T,Q)=\frac{1}{2\pi i}\int_{\lambda_{\theta}/2-iQ^{2}}^{\lambda_{\theta}/2+iQ^{2}}H_{b,c}(s)Q^{s}\frac{ds}{s},\ \ \ \ \ \  E_{2,m}(T,Q)=\frac{1}{2\pi i}\int_{\lambda_{\theta}/2+\xi_{m}iQ^{2}}^{1/\theta+\delta_{Q}+\xi_{m}iQ^{2}}H_{b,c}(s)Q^{s}\frac{ds}{s},$$ and $\xi_{m}=(-1)^{m}.$ 
We observe that an identical argument to the one utilised in the proof of Lemma \ref{lem6.1} enables one to estimate the above integrals in a completely analogous manner, thereby obtaining
$$E_{1}(T,Q)\ll TQ^{-1/2}(\log Q)^{2},\ \ \ \ \ \ \ \ \ \ \ \ \ \ E_{2,m}(T,Q)\ll TQ^{-4/3},$$ which combined with the above equations and Lemma \ref{lem6.2} in conjunction with the discussion right above such a lemma yields
\begin{equation*}\label{errow3}J_{4}^{\theta}(T)=D_{c,b}TQ^{-1/2+1/2c}+D_{b,c}TQ^{-1/2+1/2b}+O\big(TQ^{-1/2}(\log Q)^{2}\big).
\end{equation*} For the case $b=c$ most of the argument applies save the computation of the corresponding residues, whence the formula
$$J_{4}^{\theta}(T)=TQ^{-1/2+1/2b}(F_{1,b}(\log Q)+F_{0,b})+O\big(TQ^{-1/2}(\log Q)^{2}\big)$$ holds as well. 
\end{proof}

Before stating the main proposition concerning the asymptotic evaluation of the diagonal contribution in this particular setting, it seems convenient to recall (\ref{Brs}) and (\ref{F0r}), and to define for pairs $(r,s)\in\mathbb{N}^{2}$ the constant
\begin{equation}\label{Ars}A_{r,s}=\frac{2}{r-1}\zeta(s/2r+1/2).\end{equation}
It also seems desirable to further introduce for $r\in\mathbb{N}$ the parameters
\begin{equation}\label{E01}E_{0,r}=\frac{4\gamma}{r-1}+\frac{4}{(r-1)^{2}},\ \ \ \ \ \ \ E_{1,r}=\frac{2}{r(r-1)},\end{equation} and \begin{equation}\label{H12}H_{0,r}=E_{0,r}+F_{0,r},\ \ \ \ \ \ \ \ \ H_{1,r}=E_{1,r}+F_{1,r}.\end{equation} We define for convenience the linear polynomial
\begin{equation}\label{kris}H_{r}(x)=(2\pi)^{\theta(1/2-1/2r)}\big(\theta H_{1,r} x+H_{0,r}-\theta H_{1,r}\log 2\pi\big).\end{equation}
It also seems pertinent to recall (\ref{Cars}), and worth observing that in view of the definition (\ref{Kabc}) then it transpires that
\begin{equation}\label{K1bc}K_{1,b,c}=\zeta\big((1+b)/2\big)\zeta\big((1+c)/2\big).\end{equation}

\begin{prop}\label{prop3k}
Let $a=1$ and $b,c\in\mathbb{N}$ satisfying $1<c<b$, and let $0<\theta<1$. Then assuming, as we may, the proviso (\ref{fracc}) one has
$$I_{1,\theta}(T)=K_{1,b,c}T-C_{1,c,b}TQ^{-1/2+1/2c}-C_{1,b,c}TQ^{-1/2+1/2b}+O\big(TQ^{-1/2}(\log Q)^{2}\big).$$ If on the contrary $b=c>1$ then instead
$$I_{1,\theta}(T)=K_{1,b,b}T-\big(H_{1,b}\log Q+H_{0,b}\big)TQ^{-1/2+1/2b}+O\big(TQ^{-1/2}(\log Q)^{2}\big).$$
\end{prop}
\begin{proof}
We begin the proof by recalling (\ref{0.1212}) and observing that an argument akin to that in (\ref{assumth}) combined as is customary with (\ref{n1prima}) and (\ref{lades}) delivers
$$J_{3}^{\theta}(T)=T\sum_{r_{2}^{b}r_{3}^{c}\geq Q}r_{2}^{-(1+b)/2}r_{3}^{-(1+c)/2}.$$
We denote $\delta_{Q}=(\log Q)^{-1}$ and observe that then Perron's formula (see \cite[Theorem 5.2]{MonVau}) yields
$$J_{3}^{\theta}(T)=\frac{T}{2\pi i}\int_{\delta_{Q}-iQ^{2}}^{\delta_{Q}+iQ^{2}}\zeta\Big(\frac{1+b}{2}-bs\Big)\zeta\Big(\frac{1+c}{2}-cs\Big)Q^{-s}\frac{ds}{s}+R(T,Q),$$
wherein the reader may note, but not before recalling (\ref{pries}), that one has the inequality \begin{equation}\label{RXT}R(T,Q)\ll R_{1}(T,Q)+TQ^{-2}\ll TQ^{-1}(\log Q),\end{equation} wherein the last step we applied Lemma \ref{lem6.2} in conjunction with the discussion right above such a lemma. It seems pertinent to remark as above that the error term in equation (5.8) of \cite[Theorem 5.2]{MonVau} is $O(TQ^{-2})$ in the present context. We recall (\ref{rris}) and make the change of variables $z=-s$ for convenience to the end of deriving the analogous formula
$$J_{3}^{\theta}(T)=\frac{-T}{2\pi i}\int_{-\delta_{Q}-iQ^{2}}^{-\delta_{Q}+iQ^{2}}G_{b,c}(s)Q^{s}\frac{ds}{s}+R(T,Q).$$

We recall (\ref{Mbc}) to the reader, define $\xi_{m}=(-1)^{m}$ and shift the contour integral to the line $\text{Re}(s)=-1/2$ to obtain for the case $c<b$ the formula
\begin{align*}J_{3}^{\theta}(T)=&A_{c,b}TQ^{-1/2+1/2c}+A_{b,c}TQ^{-1/2+1/2b}-\frac{TL_{b,c}(Q)}{2\pi i}\nonumber
\\
&-\frac{1}{2\pi i}\sum_{m=1}^{2}\xi_{m}\int_{-1/2+\xi_{m}iQ^{2}}^{-\delta_{Q}+\xi_{m}iQ^{2}}G_{b,c}(s)Q^{s}\frac{ds}{s}+R(T,Q).
\end{align*}
It is apparent that a combination of both Lemma \ref{lem6.1} and equation (\ref{RXT}) suffices to deduce 
\begin{equation}\label{J3for}J_{3}^{\theta}(T)= A_{c,b}TQ^{-1/2+1/2c}+A_{b,c}TQ^{-1/2+1/2b}+O\big(TQ^{-1/2}(\log Q)^{2}\big).\end{equation}

In view of the conditions required to apply Lemmata \ref{lem6.1} and \ref{lem6.2} it transpires that in the analysis for $b=c$ the depart from its counterpart only has its reliance on the computation of the corresponding residues. One thus obtains
\begin{equation}\label{J3for3}J_{3}^{\theta}(T)= (E_{1,b}\log Q+E_{0,b})TQ^{-1/2+1/2b}+O\big(TQ^{-1/2}(\log Q)^{2}\big).\end{equation}
The proposition then follows by combining both (\ref{J3for}) and (\ref{J3for3}) with (\ref{I131}), (\ref{J1J9}) and Lemmata \ref{prop478} and \ref{lemJ4}.
\end{proof}

\section{Bounds for integrals of unimodular functions}\label{simple}
The upcoming discussion will be devoted to the application of Titchmarsh \cite[Lemmata 4.2, 4.4]{Tit} for the purpose of estimating some of the integrals involved in (\ref{piresiti2}) and (\ref{piresiti23}), it being appropiate to define for fixed $0<\theta<1$ and pairs of positive real numbers $r,s>0$ the terms
$$Y_{2,r,s}^{\theta}(T)=\int_{0}^{T}D_{\theta}(1/2+ait)D(1/2+irt)D(1/2-ist)\chi(1/2-irt)dt$$ and
$$Y_{2,r,s,a}^{\theta}(T)=\int_{0}^{T}\lvert D_{\theta}(1/2+ait)\rvert^{2}D(1/2+irt)D(1/2-ist)\chi(1/2-irt)dt.$$
Equipped with these definitions we write \begin{equation}\label{I2rs}I_{2,\theta}(T)=Y_{2,b,c}^{\theta}(T)+Y_{2,c,b}^{\theta}(T)\end{equation} and
\begin{equation*}I_{2,\theta,a}(T)=Y_{2,b,c,a}^{\theta}(T)+Y_{2,c,b,a}^{\theta}(T).\end{equation*}
Likewise, we further introduce the integral
$$I_{3,\theta}(T)=\int_{0}^{T}D_{\theta}(1/2+ait)D(1/2+bit)D(1/2+cit)\chi(1/2-bit)\chi(1/2-cit)dt,$$ and also consider $$I_{3,\theta,a}(T)=\int_{0}^{T}\lvert D_{\theta}(1/2+ait)\rvert^{2}D(1/2+bit)D(1/2+cit)\chi(1/2-bit)\chi(1/2-cit)dt.$$
\begin{lem}\label{lem604}
Let $a,r,s>0$ be real numbers with the property that $\min(r,s)> 2\theta a$ or $\min(r,s)=a$ and $\theta=1/2$. Then 
$$Y_{2,r,s}^{\theta}(T)\ll T^{1/2}Q^{1/2},\ \ \ \ \ \ \ \ Y_{2,r,s,a}^{\theta}(T)\ll T^{1/2}Q.$$ Moreover, whenever $a\leq c\leq b$ one has
$$\max\big(I_{2,\theta}(T),I_{3,\theta}(T)\big)\ll T^{1/2}Q^{1/2},\ \ \ \ \ \ \ \ \ \ I_{3,\theta,a}(T)\ll T^{1/2}Q.$$
\end{lem}
\begin{proof}

It might be worth noting that recalling (\ref{I2rs}) and applying \cite[Lemma 5.1]{Pli} delivers the estimates for $Y_{2,r,s}^{\theta}(T)$, $I_{2,\theta}(T)$ and $I_{3,\theta}(T)$. The starting point for the analysis of $Y_{2,r,s,a}^{\theta}(T)$ shall consist of an application of the approximation formula for $\chi(1/2-rit)$ contained in Lemma \ref{lem601} to obtain
\begin{equation*}\label{prij}Y_{2,r,s,a}^{\theta}(T)=e^{-i\pi/4}\sum_{\substack{n_{4}\leq Q\\ \bfn\in \mathcal{B}_{a,r,s,\theta}}}(P_{\bfn}n_{4})^{-1/2}\int_{N_{\bfn,n_{4},\theta}}^{T}e^{iF_{2}(t)}dt+O(T^{\varepsilon}),\end{equation*} where the function $F_{2}(t)$ is defined by means of the formula
$$F_{2}(t)=rt\log rt-rt(\log 2\pi+1)-t\log (n_{1}^{a}n_{2}^{r}/n_{3}^{s}n_{4}^{a})$$ and the above parameter $N_{\bfn,n_{4},\theta}$ upon recalling (\ref{n1prima}) denotes \begin{equation}\label{N3N4}N_{\bfn,n_{4},\theta}=2\pi\max(n_{1,\theta}'^{1/\theta},n_{2}'^{2},n_{3}'^{2},n_{4,\theta}'^{1/\theta}).\end{equation} In the above line we implicitly employed equation (\ref{priosa}) in conjunction with Holder's inequality to deduce \begin{equation}\label{dirit}\int_{0}^{t}\lvert D_{\theta}(1/2+aiw)^{2}D(1/2+irw)D(1/2-isw)\lvert dw\ll t^{1+\varepsilon},\end{equation} such an observation combined with integration by parts and the aforementioned approximation formula delivering the desired error term. The reader may find it useful to observe that $F_{2}'(t)$ is an increasing function. In view of the fact that $r\geq a$, it then transpires that \begin{align*}F_{2}'(N_{\bfn,\theta})\geq &\big(r-2\theta a\big)\log\big(\max(n_{1,\theta}'^{1/2\theta},n_{2}',n_{3}')\big)+\frac{r}{2}\log r-\theta a\log a+s\log n_{3},\end{align*} whence utilising the same argument after (5.3) of \cite{Pli} enables one to deduce that
 $$Y_{2,r,s,a}^{\theta}(T)\ll T^{1/2}Q.$$

We employ for the analysis of $I_{3,\theta,a}(T)$ the approximation formula in Lemma \ref{lem601} and the argument after (\ref{dirit}), it leading to the error term thereof, to obtain
$$I_{3,\theta,a}(T)=-i\sum_{\bfn\in \mathcal{B}_{a,b,c,\theta}}P_{\bfn}^{-1/2}\int_{N_{\bfn,\theta}}^{T}e^{iF_{3}(t)}dt+O(T^{\varepsilon}),$$where the derivative of the function $F_{3}(t)$ is 
$$F_{3}'(t)=(b+c)\log t+b\log b+c\log c-(b+c)\log2\pi-\log(n_{1}^{a}n_{2}^{b}n_{3}^{c}/n_{4}^{a}),$$ it being desirable to avoid giving account of the definition of $F_{3}(t)$ for the sake of concission. We observe that then $F_{3}'(t)$ is monotonic and that $$F_{3}'(N_{\bfn,\theta})\geq (b+c-2\theta a)\log\big(\max(n_{1,\theta}'^{1/2\theta},n_{2}',n_{3}')\big)+\frac{b}{2}\log b+\frac{c}{2}\log c-\theta a\log a,$$  wherein the reader may find it useful to recall (\ref{n1prima}), whence in a similar fashion as in the discussion pertaining to $I_{3,\theta}(T)$ in \cite[Lemma 5.1]{Pli}, Titchmarsh \cite[Lemmata 4.2, 4.4]{Tit} yields
\begin{equation*}I_{3,\theta,a}(T)\ll T^{1/2}Q.\end{equation*}
\end{proof} 

\emph{Proof of Theorems \ref{thmabcy} and \ref{prop3k4}.} We thus conclude the proof of Theorem \ref{thmabcy}, it being a consequence of Propositions \ref{lem602z} and \ref{propAA} and the remark preceding equation (\ref{fracc}) combined with equation (\ref{piresiti2}) and the above lemma. Theorem \ref{prop3k4} follows instead by utilising Proposition \ref{prop3k} in lieu of Proposition \ref{propAA} and the remark before equation (\ref{fracc}) in conjunction with equation (\ref{piresiti2}) and the above lemma.

\section{The instance $a=c=1$}\label{contour}
The upcoming section shall be devoted to the analysis of $M_{1,b,1}^{\theta}(T)$. It seems worth starting such an endeavour by computing the diagonal contribution in the same vein as above, it being required to such an end to recall (\ref{NNN}) and (\ref{Babc}) and write for each $0<\theta<1$ as in (\ref{I131}) the formula
\begin{equation}\label{I1TTT}I_{1,\theta}(T)=S_{b}^{\theta}(T)T+J_{2}^{\theta}(T)-J_{4}^{\theta}(T),\end{equation} where $$S_{b}^{\theta}(T)=\sum_{\substack{\bfn\in \mathcal{B}_{1,b,1,\theta}\\  n_{1}=n_{2}^{b}n_{3}}}P_{\bfn}^{-1/2}$$ and $J_{2}^{\theta}(T)$ and $J_{4}^{\theta}(T)$ were defined in (\ref{0.121}) and (\ref{0.1212}) respectively. We further observe that the solutions of the underlying equation in both $S_{b}^{\theta}(T)$ and $J_{4}^{\theta}(T)$ can be parametrized by means of the expressions $$n_{1}=m_{3}m_{2}^{b},\ \ \ \ n_{2}=m_{2},\ \ \ \ n_{3}=m_{3},$$ whence it is then apparent that
\begin{align}\label{Sabab}S_{b}^{\theta}(T)=&
\sum_{\substack{m_{2}^{b}m_{3}\leq Q}}m_{3}^{-1}m_{2}^{-(1+b)/2}.
\end{align}
We write for further convenience \begin{equation*}\label{prrao}\gamma_{1}=\frac{1}{2}\lim_{s\to 1}((s-1)\zeta(s))''.\end{equation*} We also recall the constants $C_{1,r,s}$ and $D_{r,s}$ introduced in (\ref{Cars}) and (\ref{Brs}) respectively and define 
\begin{equation*}\label{Ebb}E_{b}=\zeta\Big(\frac{b+1}{2}\Big)\gamma+b\zeta'\Big(\frac{b+1}{2}\Big),\end{equation*}
and the universal parameters
\begin{equation}\label{c1c2c3e} c_{1}=\theta(2\gamma-F_{1,1})-\theta^{2}\log 2\pi ,\ \ \ \ \ c_{0}=\gamma^{2}+2\gamma_{1}-F_{0,1}-\theta(2\gamma-F_{1,1})\log 2\pi +\frac{\theta^{2}}{2}(\log 2\pi)^{2},\end{equation} the constants $F_{0,1}$ and $F_{1,1}$ having been defined in (\ref{F0r}). It further seems convenient to consider the linear polynomial
\begin{equation}\label{Kb}K_{b}(x)=\theta\zeta\Big(\frac{b+1}{2}\Big)x+ E_{b}-D_{b,1}-\theta\zeta\Big(\frac{b+1}{2}\Big)\log 2\pi. \end{equation}
\begin{prop}\label{euse}
Let $a=c=1$ and $b>1$, and let $\theta\leq 1/2.$ Then assuming as we may the condition (\ref{fracc}) one has that
$$I_{1,\theta}(T)=TK_{b}(\log T)-C_{1,b,1}T^{1-\theta(1/2-1/2b)}+O\big(T^{1-\theta/2}(\log T)^{2}\big).$$ If instead $b=1$ then 
$$I_{1,\theta}(T)=\frac{\theta^{2}}{2}T(\log T)^{2}+c_{1}T\log T+c_{0}T+O\big(T^{1-\theta/2}(\log T)^{2}\big).$$

\end{prop}

\begin{proof}
We set $\delta_{Q}=(\log Q)^{-1}$ and utilise the formula in (\ref{Sabab}) for $S_{b}^{\theta}(T)$ in conjunction with Perron's formula as in the discussion in (\ref{RXT}) and an argument similar to that in (\ref{assumth}) combined as is customary with (\ref{n1prima}) and (\ref{lades}) to obtain
\begin{align*}S_{b}^{\theta}(T)&=\sum_{\substack{m_{2}^{b}m_{3}\leq Q}}m_{3}^{-1}m_{2}^{-(1+b)/2}=\frac{1}{2\pi i}\int_{\delta_{Q}-iQ^{2}}^{\delta_{Q}+iQ^{2}}G_{b,1}(s)Q^{s}\frac{ds}{s}+R_{S}(Q),
\end{align*}
the function $G_{b,1}(s)$ having been defined in (\ref{rris}), wherein the reader may recall from (\ref{pries}) that a customary application of Lemma \ref{lem6.2} yields $$R_{S}(Q)\ll Q^{-1}(\log Q).$$
We recall (\ref{Ars}) and shift the contour integral to the line $\text{Re}(s)=-1/2$ to obtain for the case $b>1$ the formula
\begin{align*}\label{linde}S_{b}^{\theta}(T)=&\zeta\big((1+b)/2\big)\log Q+E_{b}-A_{b,1}Q^{-1/2+1/2b}+\frac{1}{2\pi i}\int_{-1/2-iQ^{2}}^{-1/2+iQ^{2}}G_{b,1}(s)Q^{s}\frac{ds}{s}
\\
&+\frac{1}{2\pi i}\sum_{m=1}^{2}\xi_{m}\int_{-1/2+\xi_{m}iQ^{2}}^{\delta_{Q}+\xi_{m}iQ^{2}}G_{b,1}(s)Q^{s}\frac{ds}{s}+R_{S}(Q),
\end{align*}
wherein we recall that $\xi_{m}=(-1)^{m}.$ Therefore, an application of Lemma \ref{lem6.1} in conjunction with the above equations delivers
\begin{equation}\label{J3forih}S_{b}^{\theta}(T)=\zeta\big((1+b)/2\big)\log Q+E_{b}-A_{b,1}Q^{-1/2+1/2b}+O\big(Q^{-1/2}(\log Q)^{2}\big).\end{equation}

In view of the conditions required to apply Lemmata \ref{lem6.1} and \ref{lem6.2} it transpires that in the analysis for $b=1$ the only difference from that of $b>1$ has its reliance on the computation of the corresponding residues. One thus obtains
\begin{equation}\label{J3for34}S_{1}^{\theta}(T)=\frac{1}{2}(\log Q)^{2}+2\gamma \log Q+(\gamma^{2}+2\gamma_{1})+O\big(Q^{-1/2}(\log Q)^{2}\big).\end{equation}
The proposition then follows by combining equation (\ref{I1TTT}) and (\ref{Sabab}) and Lemma \ref{lemJ4} and \cite[Lemma 8.2]{Pli} with both (\ref{J3forih}) and (\ref{J3for34}) and the assumption $\theta\leq 1/2$, the error term $O(T^{1/2}Q^{1/2})$ stemming from the application of \cite[Lemma 8.2]{Pli} being inferior to its counterpart in the present discussion.
\end{proof}

\emph{Proof of Theorem \ref{ppop}.} We conclude the proof of Theorem \ref{ppop} by simply observing that an application of the above proposition and the remark before equation (\ref{fracc}) combined with Lemma \ref{lem604} to equation (\ref{piresiti2}) then delivers the desired result.

\section{Twisted mixed moments of the Riemann-zeta function}\label{twist}
The purpose of the present section is to utilise some of the technical results in Sections \ref{case1} and \ref{simple} to the end of asymptotically evaluate $M_{1,1,b,c}^{\theta}(T)$. As a prelude to our discussion we recall (\ref{QQQQ}), (\ref{piresiti23}) and (\ref{N3N4}) and denote
\begin{align*}\label{Is1}I_{1,\theta,1}(T)&
=\int_{0}^{T}\lvert D_{\theta}(1/2+it)\rvert^{2}D(1/2-bit)D(1/2-cit)dt
\\
&=\sum_{\substack{\bfn\in \mathcal{B}_{1,b,c,\theta}\\ n_{4}\leq Q}}(n_{4}P_{\bfn})^{-1/2}\int_{N_{\bfn,n_{4},\theta}}^{T}\Big(\frac{n_{2}^{b}n_{3}^{c}n_{4}}{n_{1}}\Big)^{it}dt,\end{align*}
 and write in a similar vein as in (\ref{I1TTT}) the formula
\begin{equation}\label{I1TTTrs}I_{1,\theta,1}(T)=S_{b,c}^{\theta}(T)T+J_{2,1}^{\theta}(T)-J_{4,1}^{\theta}(T),\end{equation} wherein $$S_{b,c}^{\theta}(T)=\sum_{\substack{\bfn\in \mathcal{B}_{1,b,c,\theta}\\  n_{1}=n_{2}^{b}n_{3}^{c}n_{4}\\ n_{4}\leq Q}}(n_{4}P_{\bfn})^{-1/2},\ \ \ \ \ \ J_{4,1}^{\theta}(T)= \sum_{\substack{\bfn\in \mathcal{B}_{1,b,c,\theta}\\  n_{1}=n_{2}^{b}n_{3}^{c}n_{4}\\ n_{4}\leq Q}}N_{\bfn,n_{4},\theta}(n_{4}P_{\bfn})^{-1/2},$$ and 
$$J_{2,1}^{\theta}(T)=\sum_{\substack{\bfn\in \mathcal{B}_{1,b,c,\theta}\\ n_{1}\neq n_{2}^{b}n_{3}n_{4}}}P_{\bfn}^{-1/2}\int_{N_{\bfn,n_{4},\theta}}^{T}\Big(\frac{n_{2}^{b}n_{3}^{c}n_{4}}{n_{1}}\Big)^{it}dt.$$ 

The following lemma shall be devoted to analyse first $J_{2,1}^{\theta}(T)$.

\begin{lem}\label{lemi12}
Let $b,c\in\mathbb{N}$. Then upon recalling (\ref{QQQQ}) to the reader and for $0<\theta<1$ one has that
$$J_{2,1}^{\theta}(T)\ll T^{1/2}Q.$$
\end{lem}
\begin{proof}
In view of the above definitions it is apparent that
$$J_{2,1}^{\theta}(T)\ll \sum_{\substack{\bfn\in \mathcal{B}_{1,b,c,\theta}\\ n_{4}\leq Q}}\frac{(n_{4}P_{\bfn})^{-1/2}}{\big\lvert\log\big(n_{2}^{b}n_{3}^{c}n_{4}/n_{1}\big)\big\rvert}.$$
It transpires that the contribution to the above sum of tuples satisfying 
$$\big\lvert\log\big(n_{2}^{b}n_{3}^{c}n_{4}/n_{1}\big)\big\rvert\geq C $$ for any fixed constant $C>0$ is $O(T^{1/2}Q)$, whence by summing in the complementary set of tuples over the difference $r=n_{2}^{b}n_{3}^{c}n_{4}-n_{1}$ one obtains
\begin{align*}J_{2,1}^{\theta}(T)&
\ll T^{1/2}Q+(\log T)\sum_{n_{2}^{b}n_{3}^{c}n_{4}\ll Q}n_{2}^{b/2-1/2}n_{3}^{c/2-1/2}
\\
&\ll T^{1/2}Q+Q(\log T)\sum_{n_{2}^{b}n_{3}^{c}\ll Q}n_{2}^{-(b+1)/2}n_{3}^{-(c+1)/2}\ll T^{1/2}Q,
\end{align*} as desired.
\end{proof}
It seems pertinent before completing the computation pertaining to $I_{1,\theta,1}(T)$ to recall equations (\ref{Cars}), (\ref{Brs}), (\ref{F0r}), (\ref{Ars}), (\ref{E01}), (\ref{H12}) and (\ref{K1bc}), and to introduce for $(r,s)\in\mathbb{N}^{2}$ the constants
\begin{equation}\label{ebc}B_{r,s}=r\zeta'\big((1+r)/2\big)\zeta\big((1+s)/2\big)+s\zeta'\big((1+s)/2\big)\zeta\big((1+r)/2\big)+K_{1,r,s}\gamma,\end{equation}
\begin{equation}\label{abf}A_{r,s}'=\zeta(1/2+1/2r)A_{r,s},\ \ \ \ \ D_{r,s}'=\zeta(1/2+1/2r)D_{r,s},\ \ \ \ \ C_{1,r,s}'=\zeta(1/2+1/2r)C_{1,r,s},\end{equation}
\begin{equation*}E'_{0,r}=\zeta(1/2+1/2r)E_{0,r}+\frac{2\zeta'(1/2+1/2r)}{r(r-1)},\ \ \ \ \ \ \ \ \ \ \ \ \ E'_{1,r}=\zeta(1/2+1/2r)E_{1,r},\end{equation*}
\begin{equation*}F'_{0,r}=\zeta(1/2+1/2r)F_{0,r}+\frac{2\zeta'(1/2+1/2r)}{r(r\lambda_{\theta}+1)},\ \ \ \ \ \ \ \ \ \ \ \ \ F'_{1,r}=\zeta(1/2+1/2r)F_{1,r},\end{equation*}and
\begin{equation*} H'_{0,r}=E'_{0,r}+F'_{0,r},\ \ \ \ \ \ \ \ \ \ \ \ \ H'_{1,r}=E'_{1,r}+F'_{1,r}.\end{equation*}
We shall conclude our sequel of definitions by considering 
\begin{equation}\label{Polyn}
G_{r}(x)=(2\pi)^{\theta(1/2-1/2r)}\big(\theta H'_{1,r} x+H'_{0,r}-\theta H'_{1,r}\log 2\pi\big),
\end{equation}
and by introducing the linear polynomial
\begin{equation}\label{linepo}
K_{r,s}(x)=\theta K_{1,r,s}x+B_{r,s}-\theta K_{1,r,s}-\theta K_{1,r,s}\log 2\pi.\end{equation}

\begin{prop}
Let $b,c\in\mathbb{N}$ satisfying $1<c<b$ and let $\theta\in\mathbb{R}$ be a fixed parameter having the property that $0<\theta\leq 1/2$. Then assuming as we may the condition (\ref{fracc}) one has that
\begin{align*}I_{1,\theta,1}(T)=&TK_{b,c}(\log T)-C_{1,c,b}'T^{1-\theta(1/2-1/2c)}-C_{1,b,c}'T^{1-\theta(1/2-1/2b)}\\
&+O\big(T^{1-\theta/2}(\log T)^{13/4}+T^{1/2+\theta}\big).
\end{align*}
If instead $b=c>1$ then it transpires that
\begin{align*}I_{1,\theta,1}(T)=&TK_{b,b}(\log T)-G_{b}(\log T)T^{1-\theta(1/2-1/2b)}+O\big(T^{1-\theta/2}(\log T)^{13/4}+T^{1/2+\theta}\big).
\end{align*}
\end{prop}
\begin{proof}

We begin our discussion by recalling (\ref{QQQQ}) and computing $J_{4,1}^{\theta}(T)$, it being convenient to such an end to define the meromorphic function 
$$H_{1,b,c}(s)=\zeta(s+1-1/\theta)\zeta\big(bs+(1+b)/2-b/\theta\big)\zeta\big(cs+(1+c)/2-c/\theta\big)$$ and take $\delta_{Q}=(\log Q)^{-1}$. We then recall equation (\ref{sthe}) and note as in the discussion right above (\ref{J3111}) that the tuples underlying the sum in the definition of $J_{4,1}^{\theta}(T)$ satisfy
$$n_{1}^{1/\theta}=n_{2}^{b/\theta}n_{3}^{c/\theta}n_{4}^{1/\theta}\geq \max(n_{2}^{2}/b,n_{3}^{2}/c,n_{4}^{1/\theta}),$$ the above inequality being in turn a consequence of the assumption on $\theta$, whence it transpires upon recalling (\ref{N3N4}) that $N_{\bfn,n_{4},\theta}=2\pi n_{1}^{1/\theta}.$ 

We then observe that an application of Perron's formula (see \cite[Theorem 5.2]{MonVau}) yields
\begin{align*}J_{4,1}^{\theta}(T)=&2\pi \sum_{r_{2}^{b}r_{3}^{c}r_{4}\leq Q}r_{2}^{b/\theta-(b+1)/2}r_{3}^{c/\theta-(c+1)/2}r_{4}^{1/\theta-1}=i^{-1}\int_{1/\theta+\delta_{Q}-iQ^{2}}^{1/\theta+\delta_{Q}+iQ^{2}}H_{1,b,c}(s)Q^{s}\frac{ds}{s}
\\
&+E_{-1}(T,Q),
\end{align*}
wherein $E_{-1}(T,Q)$ satisfies as is customary the estimate
$$E_{-1}(T,Q)\ll R_{1,1}(T,Q)+TQ^{-2}\ll TQ^{-1}(\log Q)^{2},$$ the term $R_{1,1}(T,Q)$ having been previously defined in (\ref{pries}), and wherein the last step we applied Lemma \ref{lem6.2} in conjunction with the discussion right above such a lemma. We focus first on the instance $c<b$, recall (\ref{K1bc})  and routinarily shift the line of integration to $\text{Re}(s)=\lambda_{\theta}/2$ in the above integral to obtain
\begin{align*}J_{4,1}^{\theta}(T)=&\theta K_{1,b,c}T+D'_{c,b}TQ^{-1/2+1/2c}+D'_{b,c}TQ^{-1/2+1/2b}+E_{-1}(T,Q)+2\pi E_{1,-1}(T,Q)
\\
&+2\pi\sum_{m=1}^{2}\xi_{m}E_{2,m,-1}(T,Q),
\end{align*}
wherein we wrote $$E_{1,-1}(T,Q)=\frac{1}{2\pi i}\int_{\lambda_{\theta}/2-iQ^{2}}^{\lambda_{\theta}/2+iQ^{2}}H_{1,b,c}(s)Q^{s}\frac{ds}{s}$$ and $$E_{2,m,-1}(T,Q)=\frac{1}{2\pi i}\int_{\lambda_{\theta}/2+\xi_{m}iQ^{2}}^{1/\theta+\delta_{Q}+\xi_{m}iQ^{2}}H_{1,b,c}(s)Q^{s}\frac{ds}{s},\ \ \ \ \ \ \ \ \ \ \ \ \ m=1,2$$
for $\xi_{m}=(-1)^{m}.$ In the interest of curtailing our discussion it seems worth recalling (\ref{QQQQ}) and observing that an analogous argument to the one employed in the proof of Lemma \ref{lem6.1} allows one to bound the above integrals similarly, thereby obtaining
$$E_{1,-1}(T,Q)\ll TQ^{-1/2}(\log Q)^{13/4},\ \ \ \ \ \ \ \ \ \ \ \ \ \ E_{2,m,-1}(T,Q)\ll TQ^{-1},$$ which combined with the above equations and Lemma \ref{lem6.2} in conjunction with the discussion right above such a lemma yields
\begin{align*}J_{4,1}^{\theta}(T)=&\theta K_{1,b,c}T+D'_{c,b}TQ^{-1/2+1/2c}+D'_{b,c}TQ^{-1/2+1/2b}+O\big(TQ^{-1/2}(\log Q)^{13/4}\big).
\end{align*}
For the case $b=c$ the formula
$$J_{4,1}^{\theta}(T)=\theta K_{1,b,b}T+\big(F'_{1,b}(\log Q)+F'_{0,b}\big)TQ^{-1/2+1/2b}+O\big(TQ^{-1/2}(\log Q)^{13/4}\big)$$ holds instead. 

The examination of the term $S_{b,c}^{\theta}(T)$ requires recalling first (\ref{rris}). We set $\delta_{Q}=(\log Q)^{-1}$ and employ Perron's formula to obtain
\begin{align*}S_{b,c}^{\theta}(T)&=\sum_{\substack{r_{2}^{b}r_{3}^{c}r_{4}\leq Q}}r_{2}^{-(1+b)/2}r_{3}^{-(1+c)/2}r_{4}^{-1}=\frac{1}{2\pi i}\int_{\delta_{Q}-iQ^{2}}^{\delta_{Q}+iQ^{2}}G_{1,b,c}(s)Q^{s}\frac{ds}{s}+R_{S,1}(Q),
\end{align*}
the function $G_{1,b,c}(s)$ having been defined in (\ref{rris}), and wherein the reader may recall from (\ref{priest}) that a customary application of Lemma \ref{lem6.2} yields $$R_{S,1}(Q)\ll Q^{-1}(\log Q)^{2}.$$
We recall equations (\ref{Ars}), (\ref{ebc}) and (\ref{abf}), and shift the contour integral to the line $\text{Re}(s)=-1/2$ to obtain for the case $b>1$ the formula
\begin{align*}\label{linde}S_{b,c}^{\theta}(T)=&K_{1,b,c}\log Q+B_{b,c}-A'_{b,c}Q^{-1/2+1/2b}-A'_{c,b}Q^{-1/2+1/2c}
\\
&+\frac{1}{2\pi i}\int_{-1/2-iQ^{2}}^{-1/2+iQ^{2}}G_{1,b,c}(s)Q^{s}\frac{ds}{s}+\frac{1}{2\pi i}\sum_{m=1}^{2}\xi_{m}\int_{-1/2+\xi_{m}iQ^{2}}^{\delta_{Q}+\xi_{m}iQ^{2}}G_{1,b,c}(s)Q^{s}\frac{ds}{s}
\\
&+R_{S,1}(Q),
\end{align*}
wherein we remind the reader that $\xi_{m}=(-1)^{m}.$ Therefore, an application of Lemma \ref{lem6.1} in conjunction with the above formulae delivers
\begin{equation*}S_{b,c}^{\theta}(T)=K_{1,b,c}\log Q+B_{b,c}-A'_{b,c}Q^{-1/2+1/2b}-A'_{c,b}Q^{-1/2+1/2c}+O\big(Q^{-1/2}(\log Q)^{13/4}\big).\end{equation*}
If on the contrary $b=c$ then one has instead that
\begin{align*}\label{J3for34w}S_{b,b}^{\theta}(T)=&K_{1,b,b}\log Q+B_{b,b}-(E'_{1,b}\log Q+E'_{0,b})Q^{-1/2+1/2b}+O\big(Q^{-1/2}(\log Q)^{13/4}\big).\end{align*}
The desired result follows by the above equations combined with Lemma \ref{lemi12} and (\ref{I1TTTrs}).
\end{proof}
The proof of Theorem \ref{thm10} follows by combining the previous proposition and the remark preceding equation (\ref{fracc}) with equation (\ref{piresiti23}) and Lemma \ref{lem604}.

\end{document}